\documentclass[a4paper,11pt]{amsart}
\usepackage[hmargin={2truecm,2truecm},vmargin={3truecm,3truecm}]{geometry}

\usepackage{lineno,hyperref}
\usepackage{multirow}
\usepackage{float}
\usepackage{setspace}
\usepackage{multicol}
\usepackage{amsmath,amssymb,amsthm}
\usepackage{amsfonts}

\usepackage{color}

\def\R{\mathbb{R}}

\def\N{\mathbb{N}}
\def\F{\mathcal{F}}
\def\T{\mathcal{T}}

\def\d{\mathrm{d}}
\def\cl{\mathrm{cl}}
\def\supp{\mathrm{supp}}

\usepackage[ruled,vlined]{algorithm2e}

\newcommand{\dsum}{\displaystyle\sum}

\newtheorem{definition}{Definition}
\newtheorem{remark}{Remark}
\newtheorem{theorem}{Theorem}
\newtheorem{corollary}{Corollary}
\newtheorem{proposition}{Proposition}
\newtheorem{example}{Example}%

\usepackage{pgfplots}
\usepackage{tikz}
\pgfplotsset{compat=1.9}

\definecolor{armygreen}{rgb}{0.19, 0.49, 0.33}

\let\origmaketitle\maketitle
\def\maketitle{
  \begingroup
  \def\uppercasenonmath##1{} 
  \let\MakeUppercase\relax 
  \origmaketitle
  \endgroup
}
\begin{document}

\title[On  the  Fuzzy MCLP]{\large{On the fuzzy maximal covering location problem}}

\author{Manuel Arana-Jim\'enez}
\address{Dpt. Statistics \& OR, Universidad de C\'adiz}
\email{manuel.arana@uca.es}

\author{V\'ictor Blanco}
\address{IEMath-GR, Universidad de Granada}
\email{vblanco@ugr.es}

\author{Elena Fern\'andez}
\address{Dpt. Statistics \& OR, Universidad de C\'adiz}
\email{elena.fernandez@uca.es}

\date{\today}

\begin{abstract}
In this paper studies the maximal covering location problem, assuming imprecise knowledge of all data involved. The considered problem is modeled from a fuzzy perspective producing suitable fuzzy Pareto solutions. Some properties of the fuzzy model are studied, which validate the equivalent mixed-binary linear multiobjective formulation  proposed.  A solution algorithm is developed, based on the augmented weighted Tchebycheff method, which produces solutions of guaranteed Pareto optimality. The effectiveness of the algorithm has been tested with a series of computational experiments, whose numerical results are presented and analyzed.\\

\end{abstract}

\maketitle
\section{Introduction}

Covering location models have been extensively studied in the literature. Broadly speaking, in these problems there is a set of users with service demand, which can be satisfied by activating (opening) facilities \emph{sufficiently close} to the demand points. In particular, a given demand user will be served (covered), and its demand \emph{captured} if it is located within the coverage radius of some activated facility; that is, if its distance to some open facility does not exceed a given radius.
 In particular, if a set of facilities $S$ is opened, the demand of user $i$, $w_i$,  will be captured if an only if $d_{ij}<R_j$ for some $j\in S$, where $d_{ij}$ denotes denotes the distance from user $i$ to facility $j$ and $R_j$ is a given coverage radius, which may be different for each potential facility.	 
 The two seminal models in this area are the minimum-cover location problem and the maximal covering location problem. In the minimum-cover location problem, introduced by Toregas and ReVelle \cite{toregas71}, the objective is to find a set of facilities that covers all the demand points at minimum cardinality.  Church and ReVelle \cite{church74} proposed the maximal covering location problem (MCLP), in which there is a set-up cost for each activated facility and a budget that limits the overall set-up cost that can be incurred. Such a budget constraint reduces to a cardinality constraint if all the set-up costs are equal. The objective is to find a set of facilities that maximizes the total covered demand.\\
 Applications of covering location models arise in multiple fields and include the location of health care or emergency services, where successful service strongly depends on the distance from facilities to demand points, the location of signal-transmission facilities (TV, radio, cell-phone antennas, etc.),  where coverage is only achieved within a certain  distance from the facility, or the location of retail facilities, where the attractiveness of a facility for a potential customer clearly depends on its distance from the customer location.\\
The diversity of applications and the theoretical interest of the underlying optimization models have stimulated active research on the area in the last decades.
The interested reader is addressed to  \cite{BermanetalCOR2010} and references therein for an inspiring overview on the topic.

\emph{Coverage} and \emph{coverage radius} are two concepts inherent to covering location, which are typically subjected to several modeling assumptions.
 One of the main assumptions is that a demand user is either fully covered 
 or not covered at all. Another classical assumption is that the coverage radius that determines whether or not a demand user is covered is known. 
 However, as discussed in \cite{BermanetalCOR2010}, in many applications these assumptions may be unrealistic.
 Examples that illustrate this weakness are, for instance, the location of health care or emergency services, or the location of signal transmission facilities.
 This has motivated the study of extensions of classical models, more flexible with respect to the meaning and role of \emph{coverage} and \emph{coverage radius}.
For instance, gradual covering models mitigate the above concern by extending the {\em all-or-nothing} coverage assumption to the {\em gradual coverage}  assumption,  which is modeled by associating the coverage level of demand points with their distance to open facilities \cite{BermanetalCOR2010,BermanKrass,ChurchRoberts}.
Specifically, the captured demand of a given user is computed as $w f(d)$, where $w$ denotes the demand of the user,  $d$ its distance to the closest open facility, and $f(d)\in[0, 1]$ is a non-increasing {\em coverage function} such that $f(d)=1$ for $d\leq \delta^-$, $f(d)=0$ for $d>\delta^+$, where $0<\delta^-\le \delta^+$ are two given parameters. That is, all its demand will be served if the user is at distance at most $\delta^-$ from some open facility, and none of its demand will be served if its distance to all open facilities is greater than $\delta^+$. Otherwise a fraction of $w$ will be served, which decreases as the distance $d$ increases. Note that the MCLP is a particular case of a gradual coverage model where $\delta^-=\delta^+$ for all the users.

 A concrete aspect that contributes to make further questionable the applicability of the modeling assumptions discussed above, is that covering location models often suffer from uncertain knowledge or lack of precision on the data that define specific instances. Note that, in addition to the coverage concept and the coverage radius already mentioned, information related to users' demand, set-up costs or budgets may also be imprecise. Stochastic approaches can be suitable when uncertainty can be modeled by means of a probability distribution or a set of scenarios (see, e.g., \cite{lorena}), although it is not appropriate when the lack of precision  stems from different sources, or when the decision-maker only has an idea concerning the range for the parameters' values and a kind of belief that some values are more likely to occur than some others.
 In such cases, a fuzzy perspective seems particularly suitable for modeling the MCLP. This is the approach that we follow in this paper where we propose a fuzzy mixed-binary linear programming model to deal with the MCLP as well as a solution framework for it.

Since the seminal works in the area \cite{Zadeh,Zimmermann} fuzzy mathematical programming has been applied to address different types of optimization problems with possibilistic uncertain data, as an alternative to \emph{crisp} models where precise knowledge of data is assumed. Modeling alternatives for dealing with fuzzy entities in mathematical programming models were already discussed in \cite{Inuiguchi1990}. Approaches for handling models with integer or binary variables have also been studied (see, e.g. \cite{HerreraVerdegayEJOR,HerreraVerdegayBoolean}), including specific frameworks for some fuzzy programmes \cite{arana19}, as well as complexity results \cite{BlancoPuerto}.
In fact, the MCLP has already been studied from a fuzzy perspective. In \cite{guzman16} the authors  assume \emph{flexibility} as for the coverage, which is modeled by means of fuzzy constraints, although it is assumed that the remaining input data are precisely known.
The authors then apply a parametric approach to transform the fuzzy model into a series of parametric crisp models, which are solved using an iterated local search heuristic. In \cite{Drakulic16} the authors propose a Particle Swarm Optimization scheme to solve the MCLP when fuzzy distances and radii are considered in the problem. A different fuzzy framework for the MCLP is provided in \cite{Davari11}, in which a measure of the covered demand is maximized when the distances between the users and potential facilities are treated as fuzzy events.

In this paper we study a general model for the MCLP, assuming that imprecise knowledge is not restricted to some of the parameters and constraints, but affects to all data, namely users' demand, distances, and coverage radius, as well as to all the data referring to the budget limitations. This means that all the involved entities will be fuzzy, including parameters, constraints and the objective function. For dealing with this general MCLP we propose a novel fuzzy programme in which the decision variables that represent the coverage of users are also modeled as fuzzy numbers. Nevertheless, we will see that it is enough to consider the crisp counterpart of such variables. Following the methodology used with other problems dealing with fuzzy objectives \cite{arana19,BlancoPuerto,HerreraVerdegayEJOR}, for solving our model we operate on an equivalent mixed-binary linear multiobjective formulation. In search of {\em compromise solutions} for that problem, i.e., Pareto solutions whose objective values are as close as possible to the ideal point,
we propose a solution algorithm based on the augmented weighted Tchebycheff method, which produces solutions of guaranteed Pareto optimality  (see, e.g. \cite{marler-arora04}). For the sake of simplicity, our developments consider triangular fuzzy numbers, although our results can be extended to any fuzzy number with a finite ranking system. Finally, we have carried out  extensive computational experiments in order to study the effectiveness of the proposed solution algorithm in terms of both its computational efficiency and the quality of the solutions that it produces. The obtained results are presented and analyzed.

Summarizing, the main contributions of this paper are the following:
\begin{itemize}
	\item We consider a general version of the MCLP, in the sense that we assume imprecise knowledge affects to all data:  distances, coverage radius, users' demand, data referring to the budget limitations. We model the considered problem as a fuzzy mixed-binary linear programme, where all the involved entities are fuzzy.
\item We transform the considered programme into an equivalent mixed-binary linear multiobjective formulation, and we propose an augmented weighted Tchebycheff method for obtaining Pareto solutions for it.
	\item We consider a general budget constraint, where we assume that potential facilities at different locations may have different set-up costs, and there is a budget that limits the overall set-up cost of all the facilities that are opened. To the best of our knowledge, in the literature the budget constraint in the MCLP is modeled as a cardinality constraint on the maximum number of facilities that can be opened. That is, it is assumed that all potential facilities have the same set-up cost.
\item We carry out extensive computational experiments on benchmarking instances to evaluate the effectiveness of the proposed method and to analyze the quality of the obtained solutions. The obtained results indicate that the empirical difficulty for obtaining individual solutions of the fuzzy MCLP coincides with the difficulty of solving the classical MCLP. 
\end{itemize}

The remaining of this paper is structured as follows. With the aim of making the paper self-contained, Section \ref{sec:prelim} recalls some preliminaries of fuzzy sets and fuzzy mixed-binary linear programming that will be used in the paper. Section \ref{sec:prelim} also recalls the definition of the MCLP. Section \ref{sec:FMCLP} formally defines the fuzzy extension of the MCLP that we study. Section \ref{sec:Tcheby} develops the proposed solution algorithm based on the augmented weighted Tchebycheff method, whereas Section \ref{sec:compu} describes the computational experience and presents and analyzes the obtained results. The paper ends in Section \ref{sec:conclu} with some conclusions and comments about promising avenues for future research.


\section{Preliminaries}\label{sec:prelim}

In this section we recall the main notions and results on fuzzy sets and fuzzy integer programming that will be useful for the rest of the paper.

\subsection{Fuzzy Numbers and Nonnegative Triangular Fuzzy Numbers}\label{sec:prelim-fuzzy}

A fuzzy set on $\R^{n}$ is a mapping $\tilde{\mu}:\R^n \rightarrow \lbrack 0,1]$, called \emph{membership function}. Membership functions allow to quantify the degree of truth of the statement ``the element  $x\in\R^n$ belongs to a set $S\subseteq \R^n$''.  If $x$ clearly belongs to the desired set, one will have $\tilde{\mu}(x)=1$, whereas if it clearly does not belong to the set, one will have $\tilde{\mu}(x)=0$. In case the membership of $x$ to $S$ is not sufficiently clear, the partial membership of $x$ to $S$ is modeled by values $0 < \tilde{\mu}(x)<1$, such that the closest they are to one, the clearer it becomes that $x$ belongs to $S$. Fuzzy sets are useful to model uncertainty when it is derived from imprecision. For instance, it is usual to assume that the demand of a user for a certain service is precisely known, but in practice one may have an \textit{imprecise} interval of possible demand values, where some such values are more likely to occur than others. Standard sets, also known as \emph{crisp sets}, are
examples of fuzzy sets, since indicator functions are just a particular case of membership functions.

Any fuzzy set $\tilde{\mu}$ can be characterized by means of the so-called \emph{$\alpha$-cuts}, which are defined as follows:
	\begin{align}
[\tilde{\mu}]^{\alpha}= & \begin{cases}
\cl(\supp(\tilde{\mu})), &  \text{ if $\alpha=0$},\\
\{x \in \R^n: \tilde{\mu}(x) \geq \alpha\},                  & \text{ if $\alpha\in (0, 1]$,}
\end{cases}\nonumber
\end{align}
\noindent where $\supp(\tilde{\mu})=\{x\in \mathbb{R}^{n}$ : $\tilde{\mu}(x)>0\}$, and $\cl(A)$ is the closure of the set $A$. A fuzzy set is \textit{convex} if all its $\alpha$-cuts are convex sets.\\
A special and very useful type of fuzzy sets are fuzzy numbers. A fuzzy set on $\mathbb{R}$, $\tilde{\mu}: \R \rightarrow \lbrack 0, 1]$, is a \textit{fuzzy number} if it is normal ($[\tilde{\mu}]^1 \neq \emptyset$), upper semicontinuous, convex, and $[\tilde{\mu}]^0$ is compact. A fuzzy number $\tilde{\mu}$ is \emph{nonnegative}  if $[\tilde{\mu}]^{\alpha} \subseteq \R_+$ for all $\alpha \in [0,1]$. We will denote by $\F$ and $\F_+$  the family of all fuzzy numbers and nonnegative fuzzy numbers, respectively. Observe that the $\alpha$-cuts of fuzzy numbers are intervals of the form $
[\tilde{\mu}]^{\alpha} = \left[ \underline{{\mu}}_{\alpha }, \overline{{\mu}}_{\alpha }\right]$,
with $\underline{\mu}_{\alpha },\overline{\mu}_{\alpha }\in \mathbb{R}$, and $[\tilde{\mu}]^{\alpha_2}\subseteq [\tilde{\mu}]^{\alpha_1}$, for all $0\le\alpha_1\le\alpha_2\le 1$. Thus $\tilde{\mu}$ is a nonegative fuzzy number if and only if $\underline{{\mu}}_{0 }\geq 0$. Any crisp number $p\in \R$ can be identified with the fuzzy number whose $\alpha$-cuts are given by $\{p\}$. \\

Next we describe how to perform simple arithmetic operations with fuzzy numbers. Given $\tilde{\mu}, \tilde{\nu}\in \F$ the membership function of the sum, product by a scalar $\lambda\in \R$, and multiplication of two fuzzy numbers can be defined as:
\begin{equation*}
(\tilde{\mu}+\tilde{\nu})(z)=\sup_{z=x+y} \min \{\tilde{\mu}(x),\tilde{\nu}(y)\}; \;\; (\lambda \tilde{\mu})(z)=\left\{
	\begin{array}{ll}
		\mu\left( \frac{z}{\lambda }\right) , & \hbox{if}\text{ \ }\lambda \neq 0, \\
		0, & \hbox{if }\text{ \ }\lambda =0;%
	\end{array}%
	\right. \;\; (\tilde{\mu}\cdot\tilde{\nu})(z)=\sup_{z=x\cdot y}\min \{\tilde{\mu}(x),\tilde{\nu}(y)\}.
\end{equation*}

 In terms of $\alpha$-cuts the above operations translate into operations with closed intervals (see, e.g., \cite[Theorem 2.6]{ghaznavi16}). Specifically, for any $\tilde{\mu}, \tilde{\nu}\in \F$, $\lambda \in \R$ and $\alpha \in \lbrack 0,1]$:
\begin{eqnarray}\label{eq: sum general}
\lbrack \tilde{\mu}+ \tilde{\nu}]^{\alpha}&=&\left[ \underline{{\mu}}_{\alpha
}+\underline{{\nu}}_{\alpha }\; ,\;
\overline{{\mu}}_{\alpha
}+\overline{{\nu}}_{\alpha }\right] ,\\
\label{eq: scalar multip general}
\lbrack \lambda \tilde{\mu}]^{\alpha}&=&\left[ \min \{\lambda \underline{{\mu}}%
_{\alpha },\lambda \overline{{\mu}}_{\alpha }\},\max \{\lambda \underline{{\mu}}%
_{\alpha },\lambda \overline{{\mu}}_{\alpha }\}\right],\\ 
\label{eq: multip general}
[\tilde{\mu}\cdot \tilde{\nu}]^\alpha &=& [\mbox{min} \{\underline{{\mu}}_{\alpha}\underline{{\nu}}_{\alpha},\overline{{\mu}}_{\alpha}\overline{{\nu}}, \underline{{\mu}}_{\alpha}\overline{{\nu}}_{\alpha},\overline{{\mu}}_{\alpha}\underline{{\nu}}_{\alpha}\}, \mbox{max} \{ \underline{{\mu}}_{\alpha}\underline{{\nu}}_{\alpha},\overline{{\mu}}_{\alpha}\overline{{\nu}}, \underline{{\mu}}_{\alpha}\overline{{\nu}}_{\alpha}, \overline{{\mu}}_{\alpha}\underline{{\nu}}_{\alpha}  \}].
\end{eqnarray}

In this paper, in addition to performing basic operations, we will compare  fuzzy numbers between them using a (partial) ordering. Several definitions based on interval binary relations (see \cite{guerra12}) can be used to this end. For instance,  the well-known $LU$-fuzzy partial order (see \cite{stefanini-arana18,wu2009b})  is defined as follows. Given $\tilde{\mu}, \tilde{\nu} \in\F$, $\tilde{\mu}$ is \emph{smaller than or equal to} $\tilde{\nu}$ ($\tilde{\mu}\preceq \tilde{\nu}$) if and only if:
		\begin{equation}
		\underline{{\mu}}_{\alpha}\leq \underline{{\nu}}_{\alpha} \mbox{ and } \overline{{\mu}}_{\alpha}\leq \overline{{\nu}}_{\alpha},  \quad \forall \alpha\in [0,1].\label{order}
		\end{equation}


The relationship $\tilde{\mu}$ \emph{greater than or equal to} $\tilde{\nu}$ ($\tilde{\mu}\succeq \tilde{\nu}$) can be introduced in a similar manner. By means of the previous order relationship, we can state that $\tilde{\mu}\in\F_+$ if and only if $\tilde{\mu}\succeq 0$, where, as indicated above, the crisp number $0$ is identified with the fuzzy number whose $\alpha$-cuts are given by $\{0\}$. \\

Many families of fuzzy numbers have been used to model imprecision in different situations, e.g.,  L-R,  triangular, trapezoidal, polygonal, Gaussian, quasi-quadric, exponential, and singleton fuzzy numbers. The reader is referred to \cite{baez12,hanss05,stefanini06} for a  description of some of these families and their properties.
	Even if, in general, fuzzy numbers are characterized by infinitely many $\alpha$-cuts, some of the most popular families of fuzzy numbers can be fully described by means of a finite set of $\alpha$-cuts. In such a case, the fuzzy number is said to have a \textit{finite ranking system} (FRS). This is the case of triangular, trapezoidal or polygonal fuzzy numbers. Furthermore, other more sophisticated fuzzy numbers can be accurately approximated by fuzzy numbers with a FRS \cite{gonzalez-vila91}.\\

In this paper we will use nonnegative  fuzzy numbers with a FRS. Moreover, for the sake of simplicity, we will consider nonnegative triangular fuzzy numbers (TFNs), although all the results that we obtain are extendable to any fuzzy number with a FRS.  TFNs have been widely used because of their easy interpretation (see, for instance, \cite{Dubois78,kaufmann85,khan13,lofti09,stefanini06}), and also because they can be an initial step for more sophisticated modelling approaches~\cite{pedrycz94}. Next we give the definition of a TFN, as well as the particularization to TFNs of the simple operations described above.

 A fuzzy number $\tilde{\mu} \in \F$ is a TFN if there exist $\mu^-, \mu, \mu^+ \in \R$ such that the membership function of $\tilde{\mu}$ is given by
    $$
    \tilde{\mu}(x)=\left\{
    \begin{array}{ll}
    \frac{x-\mu^-}{{\mu}-\mu^-}, &\mbox{ if } \mu^-\le x\le {\mu},\\
    \frac{\mu^+-x}{\tilde{\mu}^-{\mu}}, &\mbox{ if } \mu< x\le \mu^+,\\
    0, & \mbox{otherwise}.
    \end{array}
    \right.
    $$
    Note that  any TFN, $\tilde{\mu}$, is completely identified with the triplet $\tilde{\mu} = (\mu^-,\mu,\mu^+)$ and that the $\alpha$-cuts of a given TFN, $\tilde{\mu}=(\mu^-,{\mu},\mu^+)$ can be easily derived as:
$$
[\tilde{\mu}]^\alpha=[\underline{\mu}_{\alpha }, \overline{\mu}_{\alpha }]=[\mu^-+\alpha({\mu}-\mu^-), \mu^+-\alpha(\mu^+-{\mu})],\quad \text{ for all } \alpha\in[0,1].
$$

 The sets of TFNs  and nonnegative TFNs are respectively denoted by $\T$ and $\T_+$. Observe also that $\T_+ = \{\tilde{\mu}=(\mu^-,\mu,\mu^+) \in \T: \mu^-\geq 0\}$.\\

Figure \ref{fig:tfn} shows the shape of a (nonnegative) TFN (left picture) and one of its $\alpha$-cuts (right picture).

\begin{figure}[h]
\begin{center}
\begin{tikzpicture}[scale=2.5]
\draw[->] (-1,0)--(2.2,0);
\draw[->] (0,-0.2)--(0,1.2);

\node at (-0.1,1) {\tiny $1$};

\draw (-0.05,1)--(0.05,1);
\draw (-0.05,0)--(0.05,0);

\draw[dashed] (1.2,0)--(1.2,1);

\draw[very thick] (-1,0)--(0.5,0);
\draw[very thick] (0.5,0)--(1.2,1);
\draw[very thick] (1.2,1)--(1.7,0);
\draw[very thick] (1.7,0)--(2,0);

\node at (0.5,-0.15) {\tiny$\mu^-$};
\node at (1.2,-0.19) {\tiny$\mu$};
\node at (1.7,-0.15) {\tiny$\mu^+$};
\end{tikzpicture}~\begin{tikzpicture}[scale=2.5]
\draw[->] (-1,0)--(2.2,0);
\draw[->] (0,-0.2)--(0,1.2);

\node at (-0.1,1) {\tiny $1$};
\node at (-0.1,0.7) {\tiny $\alpha$};

\draw (-0.05,1)--(0.05,1);
\draw (-0.05,0)--(0.05,0);

\draw[dotted] (0,0.7)--(2,0.7);
\draw[dashed] (0.9903,0)--(0.9903,0.7);
\draw[dashed] (1.35,0)--(1.35,0.7);

\draw[very thick] (0.9903,0)--(1.35,0);

\draw[thick] (-1,0)--(0.5,0);
\draw[thick] (0.5,0)--(1.2,1);
\draw[thick] (1.2,1)--(1.7,0);
\draw[thick] (1.7,0)--(2,0);

\node at (0.9903,-0.15) {\tiny$\underline{{\mu}}_{\alpha }$};

\node at (1.35,-0.12) {\tiny$\overline{{\mu}}_{\alpha }$};

\node at (0.5,-0.15) {\tiny$\mu^-$};

\node at (1.7,-0.15) {\tiny$\mu^+$};
\end{tikzpicture}
\end{center}
\caption{Shape of the triangular fuzzy number $\tilde{\mu} = (\mu^-,\hat{\mu},\mu^+)$ and one of its $\alpha$-cuts.\label{fig:tfn}}
\end{figure}
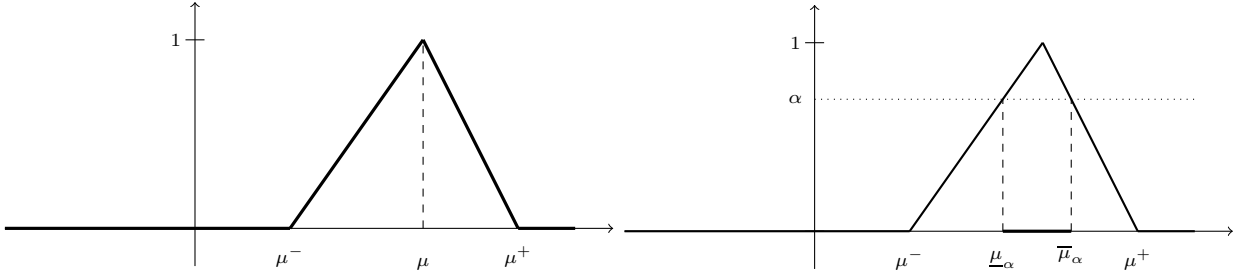

If $\tilde{\mu} = (\mu^-, \mu, \mu^+), \tilde{\nu}=(\nu^-, \nu, \nu^+) \in \T$ and $\lambda \in \R$, operations \eqref{eq: sum general} and \eqref{eq: scalar multip general} above reduce to:
\begin{eqnarray*}
        \tilde{\mu}+\tilde{\nu} & = & (\mu^-+\nu^-, {\mu}+\nu, \mu^+ + \nu^+),\label{Eq:FSum}\\[0,5em]
        \lambda \tilde{\mu}& =& \left\{
        \begin{array}{cl}
	(\lambda \mu^-,\lambda {\mu},\lambda \mu^+) & \mbox{ if }\lambda \ge 0 ,\\
	(\lambda \mu^+,\lambda {\mu},\lambda \mu^-) & \mbox{ if } \lambda  < 0.
	\end{array}
	\right.
\end{eqnarray*}

On the contrary, it is well-known that the general multiplication operator \eqref{eq: multip general} is not suitable for fuzzy numbers with a FRS, and, in particular, for TFNs. Different multiplication rules have been proposed for TFNs (see ~\cite{arana18,kaufmann85,khan13,kumar11}). In the case of nonnegative TFNs all of them coincide in the following simple expression, which we will use through this paper:
\begin{equation*}
    \label{multiplication a b nonnegative}
    \tilde{\mu}\cdot \tilde{\nu}= (\mu^-\cdot \nu^-,{\mu}\cdot{\nu},\mu^+\cdot\nu^+).
\end{equation*}

Finally, the ordering \eqref{order} can also be simplified for TFNs~\cite{arana19}, as indicated below:
\begin{equation}
\mbox{$\tilde{\mu}\preceq\tilde{\nu}$ if and only if $\mu^-\le \nu^-$, ${\mu}\le {\nu}$ and $\mu^+\le \nu^+$.}\label{AranaBlanco}
	\end{equation}

\subsection{Fuzzy Mixed-Binary Programming}\label{sec:FIP}

Fuzzy Integer Linear Programming (FILP) is widely used to address optimization problems involving linear expressions in which some of the variables are restricted to take discrete values, when there is imprecision on the information that determines the problem, and fuzzy numbers become a suitable methodology for modeling them. Depending on the characteristics of the problem and involved data, several alternatives can be used for incorporating fuzzy information within a mathematical programming model, as already discussed in \cite{Inuiguchi1990}. Explicit considerations for FILP have been discussed in \cite{HerreraVerdegayEJOR} and the complexity of these models studied in \cite{BlancoPuerto}. As a generalization of FILP, fuzzy mixed-integer linear programming (FMILP) incorporates fuzzy continuous decision variables to the formulations of the problems. FMILP results in fuzzy mixed-binary linear programming (FMBLP) when all the integer variables are restricted to take binary values. Next we briefly summarize the main concepts and results for FMBLP, which we will apply in our study of the Fuzzy Maximal Covering Location Problem.

We present the more general version of a FMBLP model where we assume that imprecise information affects to all the entities of the problem so there are fuzzy constraints, fuzzy numbers defining the constraints, fuzzy numbers in the objective function, as well as fuzzy continuous variables.  In particular a FMBLP is given by:
\begin{subequations}
\begin{align*}
{\rm (FMBLP)}\qquad \qquad \max \qquad &\;  \tilde{c}y+\tilde{w}\tilde{z}\\ 
\mbox{s.t. } \qquad & \tilde{A}y+\tilde{D}\tilde{z} \preceq \tilde{b},&&\\
& y_j\in\{0, 1\} \quad j\in J,&&\\
& \tilde{z}_{k}\in \F, \quad k\in K,
\end{align*}
\end{subequations}

\noindent where $J=\{1, 2, ..., n\}$, and $K=\{1, 2, ..., h\}$ are given index sets, $\tilde{c}=(\tilde c_{j})_{j\in J}$ and  $\tilde{w}=(\tilde w_{k})_{k\in K}$ row vectors of fuzzy numbers with a FRS, $\tilde{A}_{m\times n}=(\tilde a_{ij})_{i\in I, j\in J}$ and $\tilde{D}_{m\times h}=(\tilde d_{ik})_{i\in I, k\in K}$  matrices of  fuzzy numbers with a FRS, $\tilde b=(\tilde b_i)_{i\in I}$ a column vector of fuzzy numbers with a FRS, and $I=\{1, ...,m\}$.  Observe that $\tilde{A}y+\tilde{D}\tilde{z} \preceq \tilde{b}$ denotes a set of $m$ fuzzy constraints.

Recall that the operations between fuzzy numbers are derived using the rules described in Section \ref{sec:prelim-fuzzy} and that the ``$\max$'' operator refers to maximal solutions with respect to the partial order induced when comparing two feasible fuzzy numbers. Hence, the fuzzy objective  $\tilde{c}y+\tilde{w}\tilde{z}$, determines a fuzzy number with a FRS associated with each feasible solution, and such numbers can be ranked resorting to the solution of equivalent multiobjective optimization problems. Taking into account \eqref{AranaBlanco}, for the case that all parameters are TFNs, FMBLP is equivalent to a three-objective mixed-binary linear problem (see \cite{arana19} for details).

Therefore, usual multiobjective techniques for mixed-binary linear programming, focusing on finding Pareto solutions, can be applied for solving  FMBLP. Let us recall the concept of fuzzy Pareto solution of FMBLP.
\begin{definition}\label{def: fuzzy pareto solution}
A  feasible FMBLP solution $({Y}, {\tilde{Z}})$ is a fuzzy Pareto solution for FMBLP if there does not exist a feasible FMBLP solution $(Y', \tilde{Z'})$ such that ${\displaystyle\tilde{w} \tilde{Z}} \preceq {\displaystyle\tilde{w}\tilde{Z'}}$ and ${\displaystyle\tilde{w}\tilde{Z}}\neq {\displaystyle\tilde{w} \tilde{Z'}}$.
\end{definition}

\subsection{The Maximal Covering Location Problem.}
Next we formally define the maximal covering location problem \cite{BermanetalCOR2010,church74}, which is the crisp optimization problem that we study from a fuzzy perspective in the following sections. Let $I$ and $J$ respectively denote the index sets for the demand points and the potential locations for facilities. For each pair $(i,j)\in I\times J$, $d_{ij}\ge 0$ denotes the distance between demand point $i$ and potential facility $j$. Associated with each demand point $i\in I$ there is a service demand $w_i\ge 0$. Associated with each potential location $j\in J$ there is a set-up cost $f_j$ for activating a facility at location $j$.  There is a budget $B$ for the total set-up cost of all the activated facilities. Each facility $j\in J$, if activated, has a coverage radius $R_j$. This means that if a facility is opened at location $j\in J$ the demand of all points whose distance to $j$ does not exceed $R_j$ will be served. The MCLP is to find a set of facilities whose overall set-up cost does not exceed $B$, that maximizes the overall served demand.\\

The reader may observe that, in fact, the problem that we have defined above is more general than the one that is  usually \emph{labeled} as MCLP (see, e.g., \cite{BermanetalCOR2010}):
\begin{itemize}
\item We use a general budget constraint, which allows for facilities with different set-up costs. Classical models for the MCLP use a cardinality constraint limiting the number of constraints that can be activated. Our general budget constraint clearly reduces to a cardinality constraint when the set-up cost of all the facilities are the same.
\item In its turn, allowing for facilities with different set-up costs further increases the flexibility of the model that we consider as it permits having several candidate facilities placed at the same site, each of them with a different coverage radius and set-up cost.
\end{itemize}

Our mathematical programming formulation for the MCLP uses the same decision variables as for classical models. Specifically, we use the following sets of decision variables:

$$y_j=\left\{
\begin{array}{ll}
1 & \mbox{ if facility $j$ is actived}\\
0 & \mbox{ otherwise,}
\end{array}
\right. \quad z_i=\left\{
\begin{array}{ll}
1 & \mbox{ if demand of node $i$ is covered}\\
0 & \mbox{ otherwise,}
\end{array}
\right.$$

\noindent for all $j\in J$, $i \in I$.

Furthermore, we use we use the notation $J_i$ to represent the set of potential locations that cover the demand point $i\in I$, i.e., $J_{i}=\{j\in J: d_{ij}\le R_j\}$. The MCLP formulation is:

\begin{subequations}
\begin{align}
\mbox{(MCLP)}\, \qquad \max \qquad  & {\displaystyle\sum_{i\in I}w_i z_i} &&\label{OFunc:crisp} \\
\mbox{s.t. } \qquad & {\displaystyle\sum_{j\in J}}f_j y_j\leq B, &&\label{Budget:crisp}\\
&  z_i\leq {\displaystyle\sum_{j\in J_{i}}}y_j,\,\qquad && i\in I,\label{Cover:crisp} \\
&  z_i\in [0,1], \,\qquad && i\in I,\label{domain_z:crisp}\\ 
&  y_j\in\{0,1\},\,\qquad && j\in J. \label{domain_var_y:crisp}
\end{align}
\end{subequations}
Constraint \eqref{Budget:crisp} ensures that the overall set-up cost of the facilities that are opened does not exceed the budget $B$, whereas the set of constraints \eqref{Cover:crisp} imposes that each served demand point is in the coverage radius of some open facility. The objective \eqref{OFunc:crisp} maximizes the overall sum of the demands of the served points. The domain of the variables is defined in \eqref{domain_z:crisp}-\eqref{domain_var_y:crisp}.
Note that variables $z_i$ can been relaxed to their continuous counterpart, since the maximization objective function already guarantees their integrality for optimal solutions provided that the location $y$ variables are binary. 
Formulation \eqref{OFunc:crisp}-\eqref{domain_var_y:crisp} is thus a mixed-binary linear programme.\\

\section{The Fuzzy Maximal Covering Location Problem.}\label{sec:FMCLP}

In this section we introduce the fuzzy maximal covering location problem that we study, present a FMBLP  formulation for  it, and study some of its properties.\\

In the problem that we address, uncertainty affects all parameters, constraints and continuous variables. We assume that all parameters in MCPL are  fuzzy numbers with a FRS. In particular, for each potential facility $j\in J$,  $\tilde{f}_j$ denotes its fuzzy set-up cost and $\tilde R_j$ its fuzzy coverage radius; $\tilde{w}_i \succeq 0$ is the fuzzy demand at node $i$; $\tilde{d}_{ij}\succeq 0$ the fuzzy distance between demand point $i\in I$ and potential facility $j\in J$; and, $\tilde{B}$, the fuzzy budget for the total set-up cost of the activated facilities.
Moreover, now the index set of potential facilities covering the demand point $i\in I$ is defined as $\tilde{J}_{i}=\{j\in J:\tilde{d}_{ij}\preceq\tilde{R}_j\}$.\\

The fuzzy maximal covering location problem FMCLP is to find sets of facilities that do not violate the fuzzy budget constraint, together with a fuzzy number for the coverage of each demand point, that maximizes the overall fuzzy demand. The FMCLP has some well-known particular cases.
\begin{itemize}
	\item The FMCLP extends the MCLP recently studied by Guzm\'an et al. \cite{guzman16}, in which which uncertain values are considered for distances and the coverage radius, but all other parameters, as well as the allocation decision variables $z$ are crisp.
	\item The FMCLP also extends the gradual covering location problem (GCLP) \cite{ChurchRoberts,BermanKrass,BermanetalEJOR2003,BermanetalCOR2010}, also known as \emph{the general gradual cover decay location problem}, which aims at finding a set of $p$ facilities that maximize the total captured demand.\\
We recall that in the GCLP, the captured demand is computed as follows. Two given coverage thresholds, $\delta_i^-$ and $\delta_i^+$, $0 < \delta_i^+ \le \delta_i^-$, are associated with each demand point $i\in I$. Then, the amount of demand of user $i\in I$ that is captured when all the facilities of $S\subseteq J$ are opened is modeled as $w_if_i(\Delta_i(S))$, where $f_i(d)\in[0, 1]$ is a non-increasing {\em coverage level function}, with $f_i(d)=1$ for $d\leq \delta^-_i$, $f_i(d)=0$ for $d>\delta^+_i$, where $0<\delta^-_i\le \delta^+_i$ are two given parameters, and $\Delta_i(S) = \min_{j \in S} d_{ij}$. That is, all the demand of user $i$ will be served if $i$ is at distance at most $\delta^-_i$ from some open facility, and none of its demand will be served if its distance to all open facilities is greater than $\delta^+_i$. Otherwise a fraction of $w_i$ will be served, which decreases as the distance of $i$ to the closest open facility increases.\\ Indeed, any non-increasing coverage level function $f_i(d)\in[0, 1]$ defines a membership function. Thus,  the GCLP is a particular case of the FMCLP, where uncertainty only affects to the objective function and all other entities are crisp. Furthermore, most of the considered coverage functions in the literature, as for instance linear or \textit{step-function coverages}, are actually fuzzy numbers with a FRS.

\end{itemize}

In order to build a FMBLP formulation for the FMCLP we use the same location binary variables $y$ as in  formulation \eqref{OFunc:crisp} -\eqref{domain_var_y:crisp} for the crisp MCLP, plus a set of decision variables $\tilde{z}_i$, modeled as fuzzy numbers, associated with the demand points $i\in I$, with compact support  contained in $[0,1]$ (i.e., $[\tilde{z}_i]^0 \subseteq [0,1]$), where $\tilde{z_i}$ indicates the fuzzy coverage level of demand point $i$. A formulation for the FMCLP is therefore:
\begin{subequations}
\begin{align}
\mbox{(FMCLP)}\qquad\qquad  & \max \;\;\; {\displaystyle\sum_{i\in I}\tilde{w}_i \tilde{z}_i} \qquad\label{FMCLP 1} \\
\qquad \mbox{ s.t.} &\;\; {\displaystyle\sum_{j\in J}}\tilde{f}_j y_j\preceq \tilde{B},\label{FMCLP 2}\\
   & \tilde{z}_i\preceq {\displaystyle\sum_{j\in \tilde{J}_{i}}}y_j,\quad i\in I,\label{FMCLP 3}\\
 & [\tilde{z}_i]^0\subseteq [0,1], \quad i\in I,\label{FMCLP 4}\\
 & y_j\in\{0,1\},\quad j\in J,\label{FMCLP 5}
\end{align}
\end{subequations}
\noindent{where the fuzzy vector objective function ${\sum_{i\in I}\tilde{w}_i \tilde{z}_i}$ is a fuzzy sum of fuzzy multiplications of $\tilde{w}_i$'s and $\tilde{z}_i$'s.
	Similarly to the MCLP, we represent the variable domain by $(y, \tilde{z})\in \{0, 1\}^{|J|}\times \F^{|I|}$. 
	Note that, as mentioned in the previous section, binary vectors $y\in\{0, 1\}^{|J|}$  can be considered in fuzzy partial orders and fuzzy arithmetic in the FMCLP.\\

Below we derive a property of formulation \eqref{FMCLP 1}-\eqref{FMCLP 5}, which is essential to obtain an equivalent formulation for the FMCLP where the fuzzy decision variables $\tilde z$ can be replaced by their crisp counterpart.

\begin{proposition}\label{propo: Fuzzy}
		If $(Y, \tilde Z)$ is a fuzzy Pareto solution of FMCLP, then there exists $Z\in\R^{|I|}$, with $Z_i\in [\tilde Z_i]^0$, for all $i\in I$, such that $(Y, Z)$ is feasible for FMCLP, and $\tilde{w} \tilde{Z} =\tilde{w} {Z}$. 
\end{proposition}

\begin{proof}
	If $\tilde Z_i$ is crisp for all $i\in I$, then the result is proved. Otherwise,  we have that $[\tilde Z_k]^0=[\underline{ Z}_{k,0}, \overline{ Z}_{k,0}]$ with $\underline{ Z}_{k,0}< \overline{ Z}_{k,0}$ for some $k\in I$. Consider $(Y,\tilde Z')$, with $\tilde Z'_i=\tilde Z_i$ for all $i\ne k$, and define the fuzzy number $\tilde Z'_k$ with support $[\tilde Z'_k]^0=\{ \overline{ Z}_{k,0}\}$. We have that $\tilde Z'_k$  reduces to the crisp number $\overline{ Z}_{k,0}$. It is not difficult to see that $\tilde Z'_k$ verifies (\ref{FMCLP 3}), and then $(Y,\tilde Z')$ is feasible for FMCLP.\\
	Now, let us check that $\tilde{w} \tilde{Z}\preceq \tilde{w} {\tilde Z'}$. 
	To this end, we apply \eqref{order}
	and compare their $\alpha$-cuts. Thus, given $\alpha\in [0,1]$, from (\ref{eq: multip general}), and taking into account that the variables are non-negative,
	\begin{eqnarray*}
	& [\tilde{w}_k]^{\alpha}\times [\tilde Z_k]^{\alpha} = [\underline{ w}_{k,\alpha}\underline{ Z}_{k,\alpha},\overline{ w}_{k,\alpha}\overline{ Z}_{k,\alpha}],\\
	& [\tilde{w}_k]^{\alpha}\times [\tilde Z'_k]^{\alpha} = [\underline{ w}_{k,\alpha}\overline{ Z}_{k,0},\overline{ w}_{k,\alpha} \overline{Z}_{k,0}].
	\end{eqnarray*}
	Since $\underline{ Z}_{k,\alpha}\le \overline{ Z}_{k,\alpha}\le \overline{ Z}_{k,0}$, it follows that
	$$
	\underline{ w}_{k,\alpha}\underline{ Z}_{k,\alpha}\le
	\underline{ w}_{k,\alpha}\overline{ Z}_{k,0}
	\mbox{ and }
	\overline{ w}_{k,\alpha}\overline{ Z}_{k,\alpha}\le
	\overline{ w}_{k,\alpha} \overline{ Z}_{k,0}.
	$$
	Therefore, since the previous expression is satisfied for any $\alpha\in [0,1]$, we have that
	$$
	\tilde{w}_k \tilde Z_k\preceq 	\tilde{w}_k \tilde Z'_k,
	$$
	and then,
	$$
	\tilde{w} \tilde{Z} = \sum_{i\in I}\tilde{w}_i \tilde{Z}_i=\sum_{i\ne k}\tilde{w}_i \tilde{Z}_i+\tilde{w}_k \tilde Z_k
	\preceq \sum_{i\ne k}\tilde{w}_i \tilde{Z}'_i+\tilde{w}_k \tilde Z'_k =
	\sum_{i\in I}\tilde{w}_i {\tilde Z'}_i= \tilde{w} {\tilde Z'}.
	$$
	We iterate this process on any other index $k'\in I$ such that $\tilde Z_{k'}$ is not crisp. In the end, we get $Z\in\R^{|I|}$, with $Z_i=\tilde Z'_i$, $Z_i\in [\tilde Z_i]^0$, for all $i\in I$, such that $(Y,Z)$ is feasible for FMCLP and
		\begin{equation*}\label{eq: prop 1a}
\tilde{w} \tilde{Z} \preceq \tilde{w} {Z}.
	\end{equation*}
	By hypothesis, $(Y,\tilde Z)$ is fuzzy Pareto, which, combined with (\ref{eq: prop 1a}), implies
	\begin{equation*}
	\tilde{w} \tilde{Z}=\tilde{w} {Z}.
	\end{equation*}
	And the proof is complete.
\end{proof}

Two consequences can be derived from Proposition \ref{propo: Fuzzy}. On the one hand, we can substitute the set of fuzzy variables $\tilde z$ by a set of crisp continuous decision variables, which again will be denoted by $z$. On the other hand, similarly to the formulation for the crisp MCLP counterpart, the crisp allocation variables $z$ can be relaxed to take continuous values, since fuzzy Pareto solutions $(Y,Z)$ will take binary values.\\

In the following, for ease of presentation,  we assume that all the above fuzzy parameters belong to the set $\T$ of triangular fuzzy numbers, although our results easily extend to fuzzy numbers with a FRS. Therefore,

\begin{itemize}
\item $\tilde{f}_j=({f}_j^-,{f}_j,{f}_j^+)$, for all $j\in J$.
\item $\tilde{R}_j=(R^-_j,{R}_j,R^+_j)$, for all $j\in J$.
\item $\tilde{w}_i=(w_i^-,{w}_i,w_i^+)$, for all $i\in I$.
\item $\tilde{d}_{ij}= (d_{ij}^-,{d}_{ij},d_{ij}^+)$, for all $i\in I$, $j\in J$.
\item $\tilde{J}_i = \{ j \in J: d_{ij}^- \leq  R^-_j, d_{ij}\leq R_j, d_{ij}^+ \leq R_j^+\}$, for all $i \in I$.
\end{itemize}

Similarly to \cite{arana19}, we can formalize the relationship between  fuzzy Pareto solutions of FMCLP and Pareto solutions of a multiobjective problem as stated in the following result.

\begin{theorem}\label{Characterizarion FMCLP and MMCLP}
	$(y, \tilde{z})\in \{0, 1\}^{|J|}\times(\T)^{|I|}$ is a fuzzy Pareto solution of FMCLP if and only if $(y, z)\in\{0, 1\}^{|J|}\times\R^{3|I|}$
is a Pareto solution of the following mixed integer multiobjective problem:
\begin{subequations}
	\begin{align}
\mbox{\rm (MMCLP)}\, \qquad \max \qquad & ({\displaystyle\sum_{i\in I}w_i^- z_i},{\displaystyle\sum_{i\in I}w_i^{} {z}_i},{\displaystyle\sum_{i\in I}w_i^+ z_i}) && \label{z:MMCLP 1}\\
\mbox{subject to} \qquad & {\displaystyle\sum_{j\in J}}f_j^- y_j\leq B^-, && \label{z:MMCLP 2}\\
&{\displaystyle\sum_{j\in J}} {f}_j y_j\leq B, && \label{z:MMCLP 3}\\
&{\displaystyle\sum_{j\in J}}f_j^+ y_j\leq B^+, &&\label{z:MMCLP 4}\\
& z_i\leq {\displaystyle\sum_{j\in \tilde{J}_{i}}}y_j,\quad && i\in I,\label{z:MMCLP 5}\\
& {z}_i\le 1, \quad && i\in I,\label{z:MMCLP 9}\\
& {z}_i\ge 0, \quad && i\in I,\label{z:MMCLP 8}\\
& y_j\in\{0,1\},\quad && I\in J.\label{z:MMCLP 10}
\end{align}
\end{subequations}
\end{theorem}

%

For the sake of simplicity, we denote by $\mathcal{D} = \{(y, z) \in \{0,1\}^{|J|}\times [0,1]^{|I|}:  \sum_{j\in J} f_j^- y_j\leq B^-, \sum_{j\in J}{f}_j y_j\leq {B}, \sum_{j\in J}f_j^+ y_j\leq B^+, z_i\leq \sum_{j\in \tilde{J}_{i}} y_j\}$ the feasible domain to MMCLP defined by \eqref{z:MMCLP 2}-\eqref{z:MMCLP 10}.

\begin{remark}
As mentioned, the previous results apply analogously, not only for TFNs but also for fuzzy numbers with a FRS. The only difference is that, if the demand $\widetilde{w}$ admits a FRS with $s$ $\alpha$ cuts, then, the multiobjective problem has $s$ objective functions. The number of constraints also increases if the budget and costs have FRSs with more than $3$ $\alpha$-cuts.
\end{remark}

\begin{example}\label{ex:0}
We illustrate our methodology on the $30$ points instance provided in \cite{ms98}. We consider the same coordinates and demands as in the referenced source, but randomly generated the set-up costs from a uniform distribution $U[a, b]$ with $a=100$ and $b=1000$. The budget was fixed to $B=\frac{a+b}{2}$. The crisp formulation produced a solution with objective value $3470$ in which the set of open facilities is \{2,13,20\}. The solution is drawn in Figure \ref{fig_mcl} (left). We also generated a fuzzy version of the instance by constructing triangular fuzzy numbers for each of the parameters of the model. In all cases the \textit{center} point coincides with the crisp number, but the lower and upper extremes were randomly generated from an interval $\pm 50\%$ with respect to the central point, respectively. The evaluation of the crisp solution on the three-objective formulation is $(2550.88,3470, 4184.95)$ (observe that the crisp solution is not necessarily feasible for the multiobjective formulation because of constraints \eqref{z:MMCLP 5}). Figure \ref{fig_mcl} (center and right) depicts two Pareto solutions of the fuzzy formulation, which  open facilities $\{2,13\}$  and $\{2,20,22\}$, respectively, and whose captured fuzzy demands are the TFNs given by $(2437.80,3290,3970.53)$ and $(2392.47,3250,3971.95)$, respectively.

\begin{figure}[h]
\begin{center}
\fbox{
\begin{tikzpicture}

\definecolor{color0}{rgb}{0.12156862745098,0.466666666666667,0.705882352941177}

\begin{axis}[
hide x axis,
hide y axis,
axis equal,
]
\draw[draw=black,fill=gray, fill opacity=0.1] (axis cs:2.9,3.2) circle (0.5);
\draw[draw=black,fill=gray, fill opacity=0.1] (axis cs:3.4,3) circle (0.5);
\draw[draw=black,fill=gray, fill opacity=0.1] (axis cs:2.5,1.4) circle (0.5);

\draw[draw=gray,fill=gray,opacity=0] (axis cs:2.5,1.4) circle (0.676156874780115);

\addplot [semithick, black, mark=*, mark size=1, mark options={solid}, only marks, forget plot]
table [row sep=\\]{%
2.6	2.5 \\
};
\addplot [semithick, black, mark=*, mark size=1, mark options={solid}, only marks, forget plot]
table [row sep=\\]{%
2.4	3.3 \\
};
\addplot [semithick, black, mark=*, mark size=1, mark options={solid}, only marks, forget plot]
table [row sep=\\]{%
2.9	2.1 \\
};

\addplot [semithick, black, mark=*, mark size=1, mark options={solid}, only marks, forget plot]
table [row sep=\\]{%
1.7	5.3 \\
};

\addplot [semithick, black, mark=*, mark size=1, mark options={solid}, only marks, forget plot]
table [row sep=\\]{%
2.5	6 \\
};
\addplot [semithick, black, mark=*, mark size=1, mark options={solid}, only marks, forget plot]
table [row sep=\\]{%
2.1	2.8 \\
};
\addplot [semithick, black, mark=*, mark size=1, mark options={solid}, only marks, forget plot]
table [row sep=\\]{%
3	5.1 \\
};
\addplot [semithick, black, mark=*, mark size=1, mark options={solid}, only marks, forget plot]
table [row sep=\\]{%
1.9	4.7 \\
};
\addplot [semithick, black, mark=*, mark size=1, mark options={solid}, only marks, forget plot]
table [row sep=\\]{%
1.7	3.3 \\
};
\addplot [semithick, black, mark=*, mark size=1, mark options={solid}, only marks, forget plot]
table [row sep=\\]{%
2.2	4 \\
};

\addplot [semithick, black, mark=*, mark size=1, mark options={solid}, only marks, forget plot]
table [row sep=\\]{%
2.4	4.8 \\
};
\addplot [semithick, black, mark=*, mark size=1, mark options={solid}, only marks, forget plot]
table [row sep=\\]{%
1.7	4.2 \\
};
\addplot [semithick, black, mark=*, mark size=1, mark options={solid}, only marks, forget plot]
table [row sep=\\]{%
6	2.6 \\
};
\addplot [semithick, black, mark=*, mark size=1, mark options={solid}, only marks, forget plot]
table [row sep=\\]{%
1.9	2.1 \\
};
\addplot [semithick, black, mark=*, mark size=1, mark options={solid}, only marks, forget plot]
table [row sep=\\]{%
1	3.2 \\
};
\addplot [semithick, black, mark=*, mark size=1, mark options={solid}, only marks, forget plot]
table [row sep=\\]{%
3.4	5.6 \\
};
\addplot [semithick, black, mark=*, mark size=1, mark options={solid}, only marks, forget plot]
table [row sep=\\]{%
1.2	4.7 \\
};
\addplot [semithick, black, mark=*, mark size=1, mark options={solid}, only marks, forget plot]
table [row sep=\\]{%
1.9	3.8 \\
};
\addplot [semithick, black, mark=*, mark size=1, mark options={solid}, only marks, forget plot]
table [row sep=\\]{%
2.7	4.1 \\
};
\addplot [semithick, black, mark=*, mark size=1, mark options={solid}, only marks, forget plot]
table [row sep=\\]{%
3.2	3.1 \\
};
\addplot [semithick, black, mark=*, mark size=1, mark options={solid}, only marks, forget plot]
table [row sep=\\]{%
2.7	3.6 \\
};
\addplot [semithick, black, mark=*, mark size=1, mark options={solid}, only marks, forget plot]
table [row sep=\\]{%
2.9	2.9 \\
};
\addplot [semithick, black, mark=*, mark size=1, mark options={solid}, only marks, forget plot]
table [row sep=\\]{%
3.2	2.9 \\
};
\addplot [semithick, black, mark=*, mark size=1, mark options={solid}, only marks, forget plot]
table [row sep=\\]{%
2.6	2.5 \\
};
\addplot [semithick, black, mark=*, mark size=1, mark options={solid}, only marks, forget plot]
table [row sep=\\]{%
2.4	3.3 \\
};
\addplot [semithick, black, mark=*, mark size=1, mark options={solid}, only marks, forget plot]
table [row sep=\\]{%
3	3.5 \\
};
\addplot [semithick, black, mark=*, mark size=1, mark options={solid}, only marks, forget plot]
table [row sep=\\]{%
2.9	2.7 \\
};
\addplot [semithick, black, mark=*, mark size=1, mark options={solid}, only marks, forget plot]
table [row sep=\\]{%
2.9	2.1 \\
};
\addplot [semithick, black, mark=*, mark size=1, mark options={solid}, only marks, forget plot]
table [row sep=\\]{%
3.3	2.8 \\
};
\addplot [semithick, black, mark=*, mark size=1, mark options={solid}, only marks, forget plot]
table [row sep=\\]{%
1.7	5.3 \\
};
\addplot [semithick, black, mark=*, mark size=1, mark options={solid}, only marks, forget plot]
table [row sep=\\]{%
2.5	6 \\
};
\addplot [semithick, black, mark=*, mark size=1, mark options={solid}, only marks, forget plot]
table [row sep=\\]{%
2.1	2.8 \\
};
\addplot [semithick, black, mark=*, mark size=1, mark options={solid}, only marks, forget plot]
table [row sep=\\]{%
3	5.1 \\
};
\addplot [semithick, black, mark=*, mark size=1, mark options={solid}, only marks, forget plot]
table [row sep=\\]{%
1.9	4.7 \\
};
\addplot [semithick, black, mark=*, mark size=1, mark options={solid}, only marks, forget plot]
table [row sep=\\]{%
1.7	3.3 \\
};
\addplot [semithick, black, mark=*, mark size=1, mark options={solid}, only marks, forget plot]
table [row sep=\\]{%
2.2	4 \\
};
\addplot [semithick, black, mark=*, mark size=1, mark options={solid}, only marks, forget plot]
table [row sep=\\]{%
2.9	1.2 \\
};
\addplot [semithick, black, mark=*, mark size=1, mark options={solid}, only marks, forget plot]
table [row sep=\\]{%
2.4	4.8 \\
};
\addplot [semithick, black, mark=*, mark size=1, mark options={solid}, only marks, forget plot]
table [row sep=\\]{%
1.7	4.2 \\
};
\addplot [semithick, black, mark=*, mark size=1, mark options={solid}, only marks, forget plot]
table [row sep=\\]{%
6	2.6 \\
};
\addplot [semithick, black, mark=*, mark size=1, mark options={solid}, only marks, forget plot]
table [row sep=\\]{%
1.9	2.1 \\
};
\addplot [semithick, black, mark=*, mark size=1, mark options={solid}, only marks, forget plot]
table [row sep=\\]{%
1	3.2 \\
};
\addplot [semithick, black, mark=*, mark size=1, mark options={solid}, only marks, forget plot]
table [row sep=\\]{%
3.4	5.6 \\
};
\addplot [semithick, black, mark=*, mark size=1, mark options={solid}, only marks, forget plot]
table [row sep=\\]{%
1.2	4.7 \\
};
\addplot [semithick, black, mark=*, mark size=1, mark options={solid}, only marks, forget plot]
table [row sep=\\]{%
1.9	3.8 \\
};
\addplot [semithick, black, mark=*, mark size=1, mark options={solid}, only marks, forget plot]
table [row sep=\\]{%
2.7	4.1 \\
};
\addplot [semithick, gray, mark=*, mark size=1, mark options={solid}, only marks, forget plot]
table [row sep=\\]{%
2.9	1.2 \\
};
\addplot [semithick, gray, mark=*, mark size=1, mark options={solid}, only marks, forget plot]
table [row sep=\\]{%
3.2	3.1 \\
};
\addplot [semithick, gray, mark=*, mark size=1, mark options={solid}, only marks, forget plot]
table [row sep=\\]{%
2.9	3.2 \\
};
\addplot [semithick, gray, mark=*, mark size=1, mark options={solid}, only marks, forget plot]
table [row sep=\\]{%
2.7	3.6 \\
};
\addplot [semithick, gray, mark=*, mark size=1, mark options={solid}, only marks, forget plot]
table [row sep=\\]{%
2.9	2.9 \\
};
\addplot [semithick, gray, mark=*, mark size=1, mark options={solid}, only marks, forget plot]
table [row sep=\\]{%
3.2	2.9 \\
};
\addplot [semithick, gray, mark=*, mark size=1, mark options={solid}, only marks, forget plot]
table [row sep=\\]{%
3	3.5 \\
};
\addplot [semithick, gray, mark=*, mark size=1, mark options={solid}, only marks, forget plot]
table [row sep=\\]{%
3.3	2.8 \\
};
\addplot [semithick, gray, mark=*, mark size=1, mark options={solid}, only marks, forget plot]
table [row sep=\\]{%
3.4	3 \\
};
\addplot [semithick, gray, mark=*, mark size=1, mark options={solid}, only marks, forget plot]
table [row sep=\\]{%
2.5	1.4 \\
};
\addplot [semithick, gray, mark=*, mark size=1, mark options={solid}, only marks, forget plot]
table [row sep=\\]{%
2.9	2.7 \\
};
\end{axis}

\end{tikzpicture}}~\fbox{
\begin{tikzpicture}

\definecolor{color0}{rgb}{0.12156862745098,0.466666666666667,0.705882352941177}

\begin{axis}[
hide x axis,
hide y axis,
axis equal,
]
\draw[draw=black,fill=gray, fill opacity=0.1] (axis cs:2.9,3.2) circle (0.676156874780115);
\draw[draw=black,fill=gray, fill opacity=0.2] (axis cs:2.9,3.2) circle (0.5);
\draw[draw=black,fill=gray, fill opacity=0.3] (axis cs:2.9,3.2) circle (0.442606492813914);
\draw[draw=black,fill=gray, fill opacity=0.1] (axis cs:3.4,3) circle (0.676156874780115);
\draw[draw=black,fill=gray, fill opacity=0.2] (axis cs:3.4,3) circle (0.5);
\draw[draw=black,fill=gray, fill opacity=0.3] (axis cs:3.4,3) circle (0.442606492813914);


\draw[draw=color0,fill=color0,opacity=0] (axis cs:2.5,1.4) circle (0.676156874780115);

\addplot [semithick, black, mark=*, mark size=1, mark options={solid}, only marks, forget plot]
table [row sep=\\]{%
2.6	2.5 \\
};
\addplot [semithick, black, mark=*, mark size=1, mark options={solid}, only marks, forget plot]
table [row sep=\\]{%
2.4	3.3 \\
};

\addplot [semithick, black, mark=*, mark size=1, mark options={solid}, only marks, forget plot]
table [row sep=\\]{%
2.9	2.1 \\
};
\addplot [semithick, black, mark=*, mark size=1, mark options={solid}, only marks, forget plot]
table [row sep=\\]{%
2.1	2.8 \\
};

\addplot [semithick, black, mark=*, mark size=1, mark options={solid}, only marks, forget plot]
table [row sep=\\]{%
1.7	5.3 \\
};

\addplot [semithick, black, mark=*, mark size=1, mark options={solid}, only marks, forget plot]
table [row sep=\\]{%
2.5	6 \\
};

\addplot [semithick, black, mark=*, mark size=1, mark options={solid}, only marks, forget plot]
table [row sep=\\]{%
3	5.1 \\
};
\addplot [semithick, black, mark=*, mark size=1, mark options={solid}, only marks, forget plot]
table [row sep=\\]{%
1.9	4.7 \\
};
\addplot [semithick, black, mark=*, mark size=1, mark options={solid}, only marks, forget plot]
table [row sep=\\]{%
1.7	3.3 \\
};
\addplot [semithick, black, mark=*, mark size=1, mark options={solid}, only marks, forget plot]
table [row sep=\\]{%
2.2	4 \\
};
\addplot [semithick, black, mark=*, mark size=1, mark options={solid}, only marks, forget plot]
table [row sep=\\]{%
2.5	1.4 \\
};
\addplot [semithick, black, mark=*, mark size=1, mark options={solid}, only marks, forget plot]
table [row sep=\\]{%
2.9	1.2 \\
};
\addplot [semithick, black, mark=*, mark size=1, mark options={solid}, only marks, forget plot]
table [row sep=\\]{%
2.4	4.8 \\
};
\addplot [semithick, black, mark=*, mark size=1, mark options={solid}, only marks, forget plot]
table [row sep=\\]{%
1.7	4.2 \\
};
\addplot [semithick, black, mark=*, mark size=1, mark options={solid}, only marks, forget plot]
table [row sep=\\]{%
6	2.6 \\
};
\addplot [semithick, black, mark=*, mark size=1, mark options={solid}, only marks, forget plot]
table [row sep=\\]{%
1.9	2.1 \\
};
\addplot [semithick, black, mark=*, mark size=1, mark options={solid}, only marks, forget plot]
table [row sep=\\]{%
1	3.2 \\
};
\addplot [semithick, black, mark=*, mark size=1, mark options={solid}, only marks, forget plot]
table [row sep=\\]{%
3.4	5.6 \\
};
\addplot [semithick, black, mark=*, mark size=1, mark options={solid}, only marks, forget plot]
table [row sep=\\]{%
1.2	4.7 \\
};
\addplot [semithick, black, mark=*, mark size=1, mark options={solid}, only marks, forget plot]
table [row sep=\\]{%
1.9	3.8 \\
};
\addplot [semithick, black, mark=*, mark size=1, mark options={solid}, only marks, forget plot]
table [row sep=\\]{%
2.7	4.1 \\
};
\addplot [semithick, gray, mark=*, mark size=1, mark options={solid}, only marks, forget plot]
table [row sep=\\]{%
3.4	3 \\
};
\addplot [semithick, gray, mark=*, mark size=1, mark options={solid}, only marks, forget plot]
table [row sep=\\]{%
3.3	2.8 \\
};
\addplot [semithick, gray, mark=*, mark size=1, mark options={solid}, only marks, forget plot]
table [row sep=\\]{%
3.2	3.1 \\
};
\addplot [semithick, gray, mark=*, mark size=1, mark options={solid}, only marks, forget plot]
table [row sep=\\]{%
2.9	3.2 \\
};
\addplot [semithick, gray, mark=*, mark size=1, mark options={solid}, only marks, forget plot]
table [row sep=\\]{%
2.7	3.6 \\
};
\addplot [semithick, gray, mark=*, mark size=1, mark options={solid}, only marks, forget plot]
table [row sep=\\]{%
2.9	2.9 \\
};
\addplot [semithick, gray, mark=*, mark size=1, mark options={solid}, only marks, forget plot]
table [row sep=\\]{%
3.2	2.9 \\
};
\addplot [semithick, gray, mark=*, mark size=1, mark options={solid}, only marks, forget plot]
table [row sep=\\]{%
3	3.5 \\
};
\addplot [semithick, gray, mark=*, mark size=1, mark options={solid}, only marks, forget plot]
table [row sep=\\]{%
2.9	2.7 \\
};
\end{axis}

\end{tikzpicture}}~\fbox{
\begin{tikzpicture}

\definecolor{color0}{rgb}{0.12156862745098,0.466666666666667,0.705882352941177}

\begin{axis}[
hide x axis,
hide y axis,
axis equal,
]
\draw[draw=black,fill=gray, fill opacity=0.1] (axis cs:2.9,3.2) circle (0.676156874780115);
\draw[draw=black,fill=gray, fill opacity=0.2] (axis cs:2.9,3.2) circle (0.5);
\draw[draw=black,fill=gray, fill opacity=0.3] (axis cs:2.9,3.2) circle (0.442606492813914);
\draw[draw=black,fill=gray, fill opacity=0.1] (axis cs:2.5,1.4) circle (0.676156874780115);
\draw[draw=black,fill=gray, fill opacity=0.2] (axis cs:2.5,1.4) circle (0.5);
\draw[draw=black,fill=gray, fill opacity=0.3] (axis cs:2.5,1.4) circle (0.442606492813914);
\draw[draw=black,fill=gray, fill opacity=0.1] (axis cs:2.4,4.8) circle (0.676156874780115);
\draw[draw=black,fill=gray, fill opacity=0.2 ] (axis cs:2.4,4.8) circle (0.5);
\draw[draw=black,fill=gray, fill opacity=0.3  ] (axis cs:2.4,4.8) circle (0.442606492813914);

\addplot [semithick, black, mark=*, mark size=1, mark options={solid}, only marks, forget plot]
table [row sep=\\]{%
2.6	2.5 \\
};
\addplot [semithick, black, mark=*, mark size=1, mark options={solid}, only marks, forget plot]
table [row sep=\\]{%
2.4	3.3 \\
};

\addplot [semithick, black, mark=*, mark size=1, mark options={solid}, only marks, forget plot]
table [row sep=\\]{%
2.9	2.1 \\
};
\addplot [semithick, black, mark=*, mark size=1, mark options={solid}, only marks, forget plot]
table [row sep=\\]{%
3.3	2.8 \\
};
\addplot [semithick, black, mark=*, mark size=1, mark options={solid}, only marks, forget plot]
table [row sep=\\]{%
1.7	5.3 \\
};
\addplot [semithick, black, mark=*, mark size=1, mark options={solid}, only marks, forget plot]
table [row sep=\\]{%
3.4	3 \\
};
\addplot [semithick, black, mark=*, mark size=1, mark options={solid}, only marks, forget plot]
table [row sep=\\]{%
2.5	6 \\
};
\addplot [semithick, black, mark=*, mark size=1, mark options={solid}, only marks, forget plot]
table [row sep=\\]{%
2.1	2.8 \\
};
\addplot [semithick, black, mark=*, mark size=1, mark options={solid}, only marks, forget plot]
table [row sep=\\]{%
3	5.1 \\
};
\addplot [semithick, black, mark=*, mark size=1, mark options={solid}, only marks, forget plot]
table [row sep=\\]{%
1.9	4.7 \\
};
\addplot [semithick, black, mark=*, mark size=1, mark options={solid}, only marks, forget plot]
table [row sep=\\]{%
1.7	3.3 \\
};
\addplot [semithick, black, mark=*, mark size=1, mark options={solid}, only marks, forget plot]
table [row sep=\\]{%
2.2	4 \\
};

\addplot [semithick, black, mark=*, mark size=1, mark options={solid}, only marks, forget plot]
table [row sep=\\]{%
1.7	4.2 \\
};
\addplot [semithick, black, mark=*, mark size=1, mark options={solid}, only marks, forget plot]
table [row sep=\\]{%
6	2.6 \\
};
\addplot [semithick, black, mark=*, mark size=1, mark options={solid}, only marks, forget plot]
table [row sep=\\]{%
1.9	2.1 \\
};
\addplot [semithick, black, mark=*, mark size=1, mark options={solid}, only marks, forget plot]
table [row sep=\\]{%
1	3.2 \\
};
\addplot [semithick, black, mark=*, mark size=1, mark options={solid}, only marks, forget plot]
table [row sep=\\]{%
3.4	5.6 \\
};
\addplot [semithick, black, mark=*, mark size=1, mark options={solid}, only marks, forget plot]
table [row sep=\\]{%
1.2	4.7 \\
};
\addplot [semithick, black, mark=*, mark size=1, mark options={solid}, only marks, forget plot]
table [row sep=\\]{%
1.9	3.8 \\
};
\addplot [semithick, black, mark=*, mark size=1, mark options={solid}, only marks, forget plot]
table [row sep=\\]{%
2.7	4.1 \\
};

\addplot [semithick, gray, mark=*, mark size=1, mark options={solid}, only marks, forget plot]
table [row sep=\\]{%
3.2	3.1 \\
};
\addplot [semithick, gray, mark=*, mark size=1, mark options={solid}, only marks, forget plot]
table [row sep=\\]{%
2.9	3.2 \\
};
\addplot [semithick, gray, mark=*, mark size=1, mark options={solid}, only marks, forget plot]
table [row sep=\\]{%
2.7	3.6 \\
};
\addplot [semithick, gray, mark=*, mark size=1, mark options={solid}, only marks, forget plot]
table [row sep=\\]{%
2.9	2.9 \\
};
\addplot [semithick, gray, mark=*, mark size=1, mark options={solid}, only marks, forget plot]
table [row sep=\\]{%
3.2	2.9 \\
};
\addplot [semithick, gray, mark=*, mark size=1, mark options={solid}, only marks, forget plot]
table [row sep=\\]{%
2.5	1.4 \\
};
\addplot [semithick, gray, mark=*, mark size=1, mark options={solid}, only marks, forget plot]
table [row sep=\\]{%
2.9	1.2 \\
};
\addplot [semithick, gray, mark=*, mark size=1, mark options={solid}, only marks, forget plot]
table [row sep=\\]{%
2.4	4.8 \\
};
\addplot [semithick, gray, mark=*, mark size=1, mark options={solid}, only marks, forget plot]
table [row sep=\\]{%
3	3.5 \\
};
\addplot [semithick, gray, mark=*, mark size=1, mark options={solid}, only marks, forget plot]
table [row sep=\\]{%
2.9	2.7 \\
};
\end{axis}

\end{tikzpicture}}
\end{center}
\caption{Solutions for the crisp (left) and fuzzy (center and right) models for Example \ref{ex:0}.\label{fig_mcl}}
\end{figure}
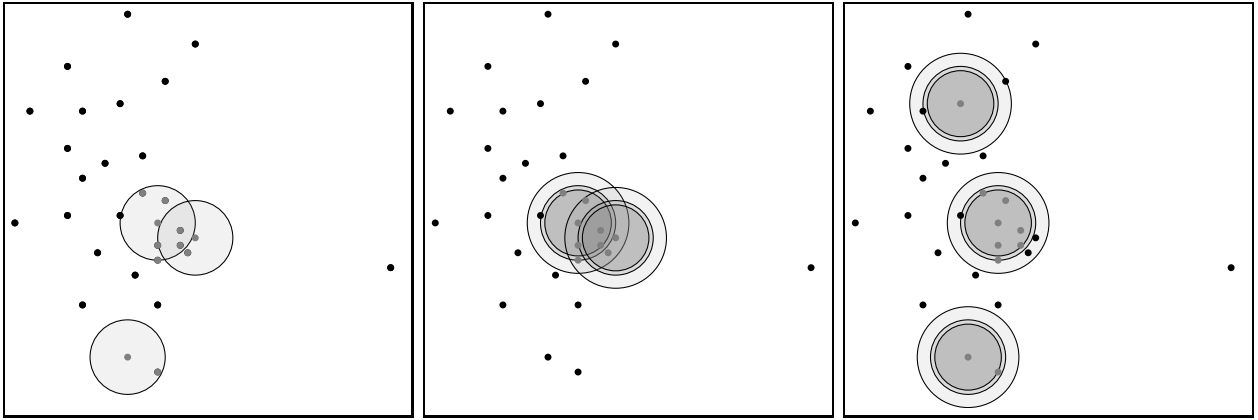

Note that the $\tilde{J}$-sets (users covered by each open facility) are different in the crisp and the fuzzy instances. In fact, the list of users covered by a facility is more restrictive under the fuzzy framework than in the crisp counterpart. This happens because the crisp sets are $J_{i}=\{j\in J: {d}_{ij}\leq {R}_j\}$, while the fuzzy sets are $\widetilde{J}_{i}=\{j\in J: {d}^-_{ij}\leq {R}^-_j, {d}_{ij}\leq {R}_j, {d}^+_{ij}\leq {R}^+_j\}$. Thus, although any feasible solution of the fuzzy problem is also feasible for the crisp problem, the opposite is not true in general.

Figure \ref{fig:ex} represents the overall demands covered in the two Pareto solutions, i.e., the TFNs indicating the \textit{served} demands for both solutions.  The gray dashed TFN is the evaluation of the crisp solution of the instance over the three objective functions of the fuzzy problem. Clearly, the two obtained solutions are incomparable with the induced fuzzy order. Observe that the crisp solution is not feasible for the fuzzy problem and the objective values of the
crisp solution are greater (component-wise) than the fuzzy solutions. The reason is that, as explained before, the fuzzy problem is more restrictive than the crisp one.
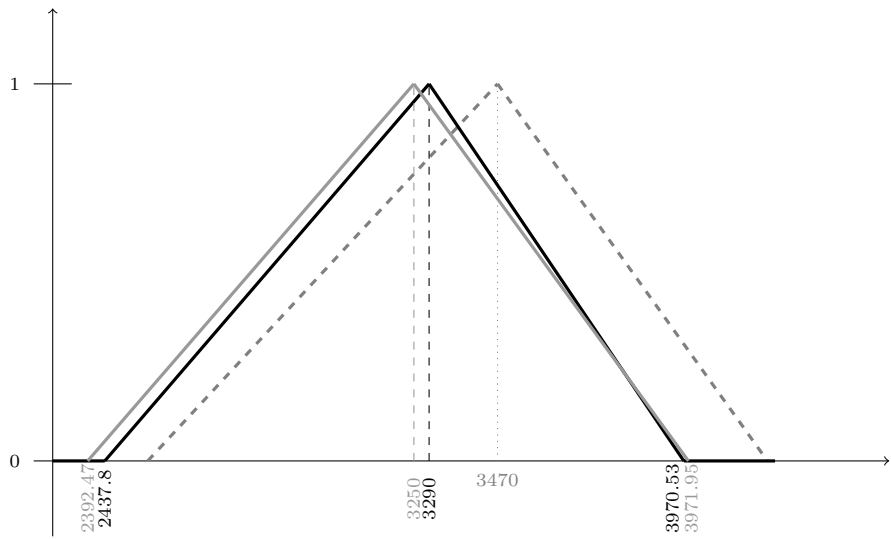
\begin{figure}[h]
\begin{center}
\begin{tikzpicture}[scale=5]
\draw[->] (2.3,0)--(4.5,0);
\draw[->] (2.3,-0.2)--(2.3,1.2);

\node at (2.2,1) {\tiny $1$};
\node at (2.2,0) {\tiny $0$};

\draw (2.25,1)--(2.35,1);
\draw (2.25,0)--(2.35,0);



\draw[dotted, gray] (3.47,0)--(3.47,1);

\draw[very thick, gray, dashed] (2.55,0)--(3.47,1);
\draw[very thick, gray, dashed] (3.47,1)--(4.18,0);
\node[gray] at (3.47,-0.05) {\tiny $3470$};


\draw[dashed] (3.29,0)--(3.29,1);

\draw[very thick] (2.437,0)--(3.29,1);
\draw[very thick] (3.29,1)--(3.96,0);

\draw[very thick] (2.3,0)--(2.437,0);
\draw[very thick] (3.96,0)--(4.2,0);


\draw[dashed, gray!80] (3.25,0)--(3.25,1);

\draw[very thick, gray!80] (2.392,0)--(3.25,1);
\draw[very thick, gray!80] (3.25,1)--(3.972,0);

\node[rotate=90]  at (2.437,-0.1) {\tiny $2437.8$}; 
\node[rotate=90]  at (3.290,-0.1) {\tiny $3290$}; 
\node[rotate=90]  at (3.93,-0.1) {\tiny $3970.53$}; 


\node[rotate=90, gray!80] at (2.392,-0.1) {\tiny $2392.47$};
\node[rotate=90, gray!80]  at (3.25,-0.1) {\tiny $3250$};
\node[rotate=90, gray!80]  at (3.98,-0.1) {\tiny $3971.95$};
\end{tikzpicture}

\end{center}
\caption{Representation of the Pareto solutions obtained in Example \ref{ex:0}.\label{fig:ex}}
\end{figure}
\end{example}

\section{Fuzzy Pareto solutions via augmented weighted Tchebycheff method}\label{sec:Tcheby}

In this section we present a solution method to generate solutions for FMCLP using the multiobjective formulation provided in the previous section. In particular, we look for fuzzy Pareto solutions \emph{close} to the ideal point. Recall that the ideal point is obtained by solving separately a single-objective problem for each of the objective functions of \eqref{z:MMCLP 1}. Specifically, let $(Y^{w^-}, Z^{w^-}), (Y^{w^{}}, Z^{w^{}}), (Y^{w^+}, Z^{w^+})\in \mathcal{D}$ be the solutions to the single-objective problems that consider separately each of the objective functions, $w^-z$, $wz$ and $w^+z$, respectively, and $F_1^I$, $F_2^I$, and $F_3^I$ their objective values. $F^I=(F_1^I,F_2^I,F_3^I)$ is usually called the \textit{ideal point} in the criterion space. In general, the point $F^I$ is not attainable by any feasible solution $(y, z)\in \mathcal{D}$ unless  $(Y^{w{^-}}, Z^{w^-})=(Y^{w^{}}, Z^{w^{}})=(Y^{w^+}, Z^{w^+})$.

\emph{Compromise solutions} consist of selecting, among the whole set of fuzzy Pareto solutions, those whose objective values are as close as possible to the ideal point.  Given a distance measure  in the objective space, $\d:\R^3\rightarrow \R_+$, compromise solutions can be obtained by finding an efficient feasible solution $(y,z)\in\mathcal{D}$, whose objective vector 
$F(z) = (F_1(z), F_2(z), F_3(z)) = (w^- z, w^{}z, w^+ z)$ is at minimum distance to $F^I$; i.e., $\d\Big(F(z),F^I\Big)$, is minimized.
Thus, a compromise solution can be found by solving the following  bilevel Compromise Solution Problem:
\begin{align}
({\rm CSP}_0) \qquad \min \qquad &\;  \d\left( F(z) , F^I\right)\label{csp0}\nonumber\\
\mbox{s.t. } \qquad & (y, z) \in \arg\max \, (w^- z, {w} {z}, w^+ z)&&\nonumber\\
& \qquad (y, z)\in\mathcal{D}.&&\nonumber
\end{align}
${\rm CSP}_0$ is a bilevel mixed-integer programming problem, whose solution is, in general, complex. The following result, whose proof is straightforward, states the equivalence of ${\rm CSP}_0$ and a single-level mixed-integer programming problem.
\begin{proposition}
${\rm CSP}_0$ is equivalent to the following single-level mixed-integer problem:
\begin{align}
({\rm CSP}) \qquad \min \qquad &\; \d\left(F(z) , F^I\right)\nonumber\\
    \mbox{s.t. }\qquad & (y, z)\in\mathcal{D}.\nonumber
\end{align}
\end{proposition}
Several types of distances can be considered for finding compromise solutions, such as those induced by $l_p$-norms, with $p\in\N\cup\{\infty\}$. In this work we use $p \in \{1, \infty\}$ so 
CSP reduces to the following mixed-integer linear programming problems:
\begin{multicols}{2}
\begin{align}
({\rm CSP}_1)\qquad\min & \dsum_{r=1}^3 \Big (F_r^I-F_r(z)\Big)\nonumber\\
\mbox{s.t.}  & \;\;(y, z)\in\mathcal{D},\nonumber
\end{align}~\begin{align}
({\rm CSP}_\infty)\qquad \min & \;\;\; \alpha\nonumber\\
\mbox{s.t.} & \;\; \alpha \geq \Big (F_r^I-F_r(z)\Big), r=1, \ldots, 3,\nonumber\\
& (y, z)\in\mathcal{D},\nonumber
\end{align}
\end{multicols}
\noindent where the absolute values involved in the definition of the $\ell_1$ and the $\ell_\infty$-norms can be avoided since the components of ideal points are always greater than or equal to the objective function value at any feasible solution.

Prior information of preferences can also be useful in the search of compromise solutions. In particular, \emph{a priori} information can be integrated within CSP via, for instance, a weighted min-max formulation also known as the weigthed Tchebycheff method. This method produces solutions that satisfy necessary conditions for Pareto optimality (see Miettinen \cite{miettinen99}), although not necessarily sufficient conditions (see Koski and Silvennoinen \cite{koski87}). This weakness is overcome in the augmented weighted Tchebycheff method. This method minimizes \emph{ad-hoc} objective functions on the original feasible domain and, for discrete problems involving linear constraints (polyhedral problems) like the MMCLP that we address, produces solutions of guaranteed Pareto optimality. Specifically, Pareto solutions can be obtained by optimizing the following objective function (see \cite{marler-arora04} and references therein for details on this method):
\begin{equation*}
H(y, z)=\max_{r\in \{1,2,3\}}\{ \lambda_r(F^{I}_r-F_r(z))  \}+\rho\sum_{r=1}^{3}(F^{I}_r-F_r(z)), \quad \text{for all } (y,z) \in \mathcal{D},
\end{equation*}
\noindent where all weights are strictly positive, with $\rho$ sufficiently small positive scalar assigned by the decision-maker. 

Taking into account the previous considerations, the following modified model is proposed for solving the MMCLP:
\begin{eqnarray*}
\mbox{(ModMMCLP)}\, & \min & \max_{r\in \{1,2,3\}}\{ \lambda_r(F^{I}_r-F_r(z))  \}+\rho\sum_{r=1}^{3}(F^{I}_r-F_r(z))\\
&\mbox{s.t.} \qquad & (y, z)\in\mathcal{D}.
\end{eqnarray*}
Pareto solutions for the MMCLP can thus be obtained by solving ModMMCLP for given strictly positive weights $(\lambda,\rho)$. Suitable simplifications allow to transform ModMMCLP to:
\begin{eqnarray*}
\mbox{(ModMMCLP')}\, & \min & \alpha +\rho\sum_{r=1}^{3}(F^{I}_r-F_r(y, z))\\
&\mbox{s.t. } & \lambda_r(F^{I}_r-F_r(z))\le\alpha,\quad r\in\{1,2,3\},\\
& \qquad & (y, z)\in\mathcal{D},\\\
& \qquad & \alpha\geq 0.
\end{eqnarray*}
\begin{proposition}
	Given a strictly positive coefficients vector $(\lambda,\rho)\in \R^{|I|+1}_+$, if $(Y,Z)\in\mathcal{D}$ is an optimal solution to ModMMCLP', then $(Y,Z)$ is a Pareto solution to FMCLP, and $(Y, \tilde{Z})\in \{0,1\}^{|J|} \times (\T)^{|I|}$ is a fuzzy Pareto solution of FMCLP.
\end{proposition}
\begin{proof}
	The proof is straightforward from \cite{marler-arora04}.
\end{proof}

\vspace*{0.6cm}

\noindent {\bf Pareto Optimality Test}\\

Let us remark that when some component or weight of $(\lambda,\rho)$ is null, the solutions produced by  ModMMCLP' are weakly Pareto, although they are not guaranteed to be Pareto solutions. In particular, note that compromise solutions for both the $\ell_\infty$-norm and the $\ell_1$-norm can be obtained from the sets of optimal solutions to ModMMCLP' for suitable values of the parameters $\lambda$ and $\rho$. In particular, the set of optimal solutions to ModMMCLP' with $(\lambda,\rho)=(1,1,1,0)$  coincides with the $\ell_\infty$-compromise solution set. For $(\lambda,\rho)=(0,0,0,1)$ we obtain the $\ell_1$-compromise solution set, while for  $(\lambda,\rho)\in \{(1,0,0,0), (0,1,0,0), (0,0,1,0) \}$, ModMMCLP' reduces to maximizing independently each objective.

In the above cases, it is possible to check whether or not the obtained solution is Pareto by means of several methods summarized in Miettinen \cite{miettinen99}, and described in Marler and Arora \cite{marler-arora04}. In particular, we have the following simple test for a given weakly Pareto solution $(Y, Z)$:
\begin{eqnarray*}
{P(Y, Z)} \qquad &\max & {\displaystyle\sum_{r=1}^{3}}\delta_r \label{prob: check Pareto}\\
&\mbox{s.t.} & F_r(y, z)-\delta_r = F_r(Y, Z),\quad r\in\{1,2,3\},\label{eq: prop check 2}\\
&& \delta_r\ge 0,\quad r\in\{1,2,3\},\nonumber\\
& \qquad & (y, z)\in\mathcal{D}. \label{prob: check Pareto_domain}
\end{eqnarray*}
If there is an optimal solution to $P(Y, Z)$ with all $\delta_r$'s at value zero, then $(Y, Z)$ is a Pareto solution for MMCPL, and its corresponding fuzzy vector given by $(Y, \tilde{Z})$ is a fuzzy Pareto solution of FMCLP. Furthermore, if $\delta_r>0$ for some $r\in\{1,2,3\}$, we can also generate a a fuzzy Pareto solution for FMCLP from an optimal solution to $P(Y, Z)$  as follows.
\begin{proposition}
	Let  $(Y, Z)$ be a weakly Pareto solution of MMCPL and $(\delta^*, Y^{*}, Z^{*})$ an optimal solution to $P(Y, Z)$. If $\delta^*_r>0$, for some $r\in\{1,2,3\}$, then $(Y^{*}, Z^{*})$ is a Pareto solution for MMCLP.
\end{proposition}
\begin{proof}
	By definition of $P(Y, Z)$, $(Y^{*}, Z^{*})$ is feasible for MMCLP. Suppose that $(Y^{*}, Z^{*})$ is not a Pareto solution for MMCLP. This implies that there exists $(\bar{Y},\bar{Z})\in\mathcal{D}$ such that $F(\bar{Y},\bar{Z})\geq F(Y^{*}, Z^{*})$ and $F(\bar{Y},\bar{Z})\neq F(Y^{*}, Z^{*})$. That is, there exists $\bar{\delta}=(\bar{\delta}_1, \bar{\delta}_2, \bar{\delta}_3)\geq 0$, with $\bar{\delta}_{r_0}>0$ for some $r_0\in\{1, 2, 3\}$, such that

	\begin{equation*}\label{eq: prop check 1}
	F_r(\bar{Y},\bar{Z})-\bar{\delta}_r= F_r(Y^{*}, Z^{*}),\quad r\in\{1,2,3\}.
	\end{equation*}
	Since  $(Y^{*}, Z^{*})$ is feasible for $P(Y,Z)$, it follows that
	\begin{equation*}\label{eq: prop check 3}
	F_r(\bar{Y},\bar{Z})-\bar{\delta}_r-\delta^*_r= F_r(Y^{*}, Z^{*})-\delta^*_r= F_r(Y, Z) ,\quad r\in\{1,2,3\}.
	\end{equation*}
	Define $\hat{\delta}=\delta^*+\bar{\delta}$. Then, $\hat{\delta}\geqq\delta^*$ with $\hat{\delta}_{r_0}>\delta^*_{r_0}$. Therefore, $(\hat{\delta},\bar{Y},\bar{Z})$ is feasible for $P(Y,Z)$ with $\sum_{r=1}^{3}\hat{\delta}_r>\sum_{r=1}^{3}\delta^*_r$, contradicting that $(\delta^*, Y^{*}, Z^{*})$ is optimal for $P(Y,Z)$.
\end{proof}
Combining the above proposition and Theorem \ref{Characterizarion FMCLP and MMCLP} we obtain the following result.
\begin{corollary}\label{cor: fuzzy Pareto solution}
		Let  $(Y, Z)$ be a weakly Pareto solution of MMCPL and $(\delta^*, Y^{*}, Z^{*})$ an optimal solution to $P(Y, Z)$. If $\delta^*_r>0$, for some $r\in\{1,2,3\}$, then $(\tilde{Y}^{*}, \tilde{Z}^{*})\in \{0, 1\}^{|J|}\times(\T)^{|I|}$ is a fuzzy Pareto solution of FMCLP.
\end{corollary}
The above results can be exploited to derive a solution algorithm for generating fuzzy Pareto solutions for FMCLP. In particular, given a set of $(\lambda,\rho)$-weights, $\Lambda \subseteq \R^4_+$, we propose the procedure described in the pseudocode of Algorithm \ref{alg}.

\begin{algorithm}[h]
  \SetAlgoLined
  \KwData{$\tilde{f}_j, \tilde{R}_j, \tilde{w}_i, \tilde{d}_{ij} \in \T_+$, for all $i\in I$, $j\in J$, $\tilde{B} \in \T_+$, $\Lambda \in \R^4_+$.}

  $\mathcal{P} = \emptyset$.\\

Compute the ideal point: $F^I=(F^{I}_{1},F^{I}_{2},F^{I}_{3})$.\\
\For{$(\lambda, \rho) \in \Lambda$}{
    Solve ModMMCLP: $(Y,Z)$ and $F=(F_1, F_2, F_3)$.\\
    \If{$F=F^I$}{$\mathcal{P} \leftarrow \mathcal{P} \cup \{(Y,  Z)\}$\\ Terminate}
    \eIf{$(\lambda,\rho) >0$}{
    $\mathcal{P} \leftarrow \mathcal{P} \cup \{(Y,Z)\}$.
    }{
    Solve $P(Y,Z)$: $(\bar \delta, \bar Y, \bar Z)$.\\
    $\mathcal{P} \leftarrow \mathcal{P} \cup \{(\bar Y, \bar Z)\}$
    }}

      \KwResult{$\mathcal{P}$.}
  \caption{Solution algorithm to generate Pareto solutions for FMCLP.\label{alg}}
\end{algorithm}
In the pseudocode, $\mathcal{P}$ denotes our solution set, and $F^I$ the ideal point, which is first computed and becomes our reference point in the criterion space to construct feasible solutions as close as possible to it. For each combination of $(\lambda,\rho)$-weights in $\Lambda$, ModMMCLP is solved. In case all the $(\lambda,\rho)$-weights are strictly positive, we are done, and the solution is added to the set $\mathcal{P}$. Otherwise, a weakly Pareto solution, $(Y,Z)$, is obtained. Since $(Y,Z)$ is not guaranteed  to be a fuzzy Pareto solution for FMCLP, the test $P(Y,Z)$  is performed. After solving $P(Y,Z)$, either $(Y,Z)$ is proved to be a fuzzy Pareto solution or the new solution produced by $P(Y,Z)$ is guaranteed to be fuzzy Pareto (Corollary \ref{cor: fuzzy Pareto solution}). After executing the algorithm for the considered set of weights we obtain a set of fuzzy Pareto solutions for FMCLP with cardinality at most $|\Lambda|$. Note that in case the ideal point is attainable, it will be found in the first iteration of the {\tt{for}} loop and we are done, since the set of solutions, $\mathcal{P}$ is just the singleton composed by such an ideal point.

\section{Computational Experiments}\label{sec:compu}

In this section we describe our computational experience and summarize the obtained numerical results. We have performed a series of experiments on a set of  maximal covering location benchmark instances widely used in the literature. In particular, we consider the datasets in  \cite{lorena} and  \cite{ms98} with $30$, $324$, $402$, $500$, $708$ and $818$ points (\url{http://www.lac.inpe.br/\%7Elorena/instancias.html}). We run two different classes of experiments on these instances. First, using the original information, in which the set-up costs $f$ were fixed to one, so the budget constraints reduce to cardinality constraints, and the budgets $B$ ranging in $\{2,3,4,5,6,7,8,9,10, 15, 20\}$. Second, randomly generating $f$ from an independent normal distribution with mean $100$ and standard deviation $10$. $B$ is defined as the sum of the $p$ smallest $f_i$ values,  for values of $p$ ranging in $\{2,3,4,5,6,7,8,9,10, 15, 20\}$. This last choice allows us to test the model over general budget constraints. The triangular fuzzy numbers for the parameters in the instances were generated such that the central point of each triangular fuzzy number coincides with the original (crisp) value in the reference instance and the lower and upper extremes were randomly generated. In particular, the fuzzy numbers, $(a^-,a,a^+)$, were obtained such that $a^-$ and and $a^+$ were drawn from uniform distributions in $U[0.80a,a]$ and $U[a,1.20a]$, respectively. We generated five instances for each of the datasets. The parameters for each of the instances as well as the detailed results are available at \url{http://bit.ly/FMCInstances}.

For each instance in our testbed we have applied Algorithm \ref{alg} with nine different choices of $(\lambda,\rho)$-weights $\Lambda = \{(0,0,0,1)$, $(1, 0, 0, 0)$, $(0, 1, 0, 0)$, $(0, 0, 1, 0)$, $(1,1,1,0.001)$, $(1,1,1,0)$, $(1,1,0, 0.001)$, $(1,0,1, 0.001)$, $(0,1,1,0.001)\}$. Hence, for each combination $(n,B)$, we have tested 45 runs, i.e., nine choices of $\Lambda$ for each of the  five random instances with the combination $(n,B)$.

All the formulations were coded in Python, and solved using Gurobi  8.0 \cite{gurobi} in a Mac OSX Mojave with an Intel Core i7 processor at 3.3 GHz and 16GB of RAM.

In Tables \ref{table:1} and \ref{table:2} we report the following information referring to each group of 45 instances with fixed $n$ ($5$ randomly generated instances and $9$ $(\lambda,\rho)$-weights):
\begin{itemize}
\item CPUTime: Average computing time of Algorithm \ref{alg}, over the 45 instances in each group.
\item CPUTime-Crisp: Average computing time for solving the crisp version of the problem (over the 45 instances in each group).
\item Different Pareto: Average number of different Pareto solutions produced by Algorithm \ref{alg}  (over the 45 instances in the group).
\item CheckPareto: \% of instances (out of the 45 in each group) in which the CheckPareto produces a solution different from the one previously obtained.
\item ReachIdeal: \% of instances (out of the 45 in each group) for which the ideal point was attained.
\end{itemize}

One can observe that the required times for solving each of the instances of the fuzzy models are slightly larger than those needed to solve the crisp model. We were able to solve all the instances in small  computing CPU times. The most time consuming instance (with $n=818$ and $p=7$) was solved in $102$ seconds. As expected, these times increase with the size of the instances. We found that in most of the instances the ideal point is attained. In particular, in $74\%$ of the instances of cardinality-constrained problem and $70\%$ of the budget-constrained problem, the ideal point was found. These percentages are quite similar for the different sizes and range in $[50\%, 80\%]$ for all values of $n$. 

Also, the test to check Pareto optimality of solutions does not produce, except in a few cases, a new solution, corroborating the Pareto-optimality of the obtained solutions. Although, in principle, it could be possible to find up to nine different Pareto solutions per instance (one for each combination of weights), the average number of different solutions ranges in $[2,3]$, often obtaining the same solutions associated with different sets of weights. We also observed that the solutions obtained when solving ModMMCPL are in most of the cases fuzzy Pareto solutions, being the outcome of the test for Pareto optimality, $P(Y,Z)$, negative in the $99.5\%$ of the solved problems. This fact confirms the good quality, relative to the fuzzy problem, of the solutions produced by ModMMCPL.\\

In Tables \ref{table:3} and \ref{table:4} we report  the following information concerning the solutions produced by the fuzzy and the crisp models, for the same set of instances. We provide:
\begin{itemize}
\item \%CoveredCrisp: Average percentage of covered demand with the crisp model.
\item \%CoveredFuzzy$^-$/\%CoveredFuzzy/\%CoveredFuzzy$^+$: Average lower bound/center/upper bound on the covered demand with the fuzzy model.
\item \#OpenCrisp: Number of open plants in the crisp model. In the cardinality-constrained model, this amount coincides with $p$, so it is omitted.
\item \#OpenFuzzy: Number of open plants in the fuzzy model.

\end{itemize}

Observe that the triplets (\%CoveredFuzzy$^-$, \%CoveredFuzzy, \%CoveredFuzzy$^+$) correspond to the triangular fuzzy numbers of the obtained solutions relative to the total demand.

\begin{table}[h]
\begin{center}
{\scriptsize \begin{tabular}{|r|r|rr|rrr|}\hline
$n$ & $B$ & CPUTime & CPUTime-Crisp & Different Pareto & CheckPareto & ReachIdeal\\\hline
\hline
\multirow{11}{*}{30} & 2    &         0.0067    &         0.2601    &             1.20    & 0.00\% & 31\%\\
& 3    &         0.0049    &         0.0014    &             1.20    & 0.00\% & 31\%\\
& 4    &         0.0047    &         0.0012    &             1.00    & 0.00\% & 100\%\\
& 5    &         0.0047    &         0.0012    &             1.00    & 0.00\% & 100\%\\
& 6    &         0.0048    &         0.0012    &             1.00    & 0.00\% & 100\%\\
& 7    &         0.0049    &         0.0012    &             1.60    & 0.00\% & 14\%\\
& 8    &         0.0051    &         0.0012    &             1.60    & 0.00\% & 14\%\\
& 9    &         0.0049    &         0.0013    &             1.00    & 0.00\% & 100\%\\
& 10   &         0.0053    &         0.0013    &             1.40    & 0.00\% & 14\%\\
& 15   &         0.0042    &         0.0012    &             1.20    & 0.00\% & 31\%\\
& 20   &         0.0017    &         0.0003    &             1.00    & 0.00\% & 100\%\\
\hline
\multirow{11}{*}{324} & 2    &         0.1235    &         0.0203    &             1.00    & 0.00\% & 100\%\\
& 3    &         0.1128    &         0.0260    &             1.20    & 0.00\% & 31\%\\
& 4    &         0.1250    &         0.0252    &             1.60    & 0.00\% & 7\%\\
& 5    &         0.1630    &         0.0275    &             1.40    & 0.00\% & 14\%\\
& 6    &         0.1351    &         0.0268    &             1.20    & 0.00\% & 31\%\\
& 7    &         0.3242    &         0.0266    &             1.60    & 0.00\% & 14\%\\
& 8    &         0.1016    &         0.0258    &             1.20    & 0.00\% & 31\%\\
& 9    &         0.1021    &         0.0273    &             1.40    & 0.00\% & 14\%\\
& 10   &         0.2268    &         0.0272    &             1.40    & 0.00\% & 14\%\\
& 15   &         0.2303    &         0.0593    &             1.40    & 0.00\% & 14\%\\
& 20   &         0.6160    &         0.1205    &             1.40    & 0.00\% & 14\%\\
\hline
\multirow{11}{*}{402} & 2    &         0.1807    &         0.0250    &             1.40    & 0.00\% & 14\%\\
& 3    &         0.3971    &         0.0303    &             1.40    & 0.00\% & 14\%\\
& 4    &         0.5470    &         0.0333    &             1.60    & 0.00\% & 7\%\\
& 5    &         0.7275    &         0.0353    &             2.00    & 0.00\% & 7\%\\
& 6    &         1.1120    &         0.0341    &             1.80    & 0.00\% & 14\%\\
& 7    &         2.0108    &         0.0408    &             1.40    & 0.00\% & 31\%\\
& 8    &         1.1455    &         0.0367    &             1.60    & 0.00\% & 31\%\\
& 9    &         0.5801    &         0.0353    &             1.40    & 0.00\% & 14\%\\
& 10   &         0.7202    &         0.0368    &             1.20    & 0.00\% & 31\%\\
& 15   &         0.5010    &         0.0397    &             1.40    & 0.00\% & 14\%\\
& 20   &         1.9336    &         0.1036    &             2.20    & 0.00\% & 3\%\\
\hline
\multirow{11}{*}{500} & 2    &         0.3164    &         0.0289    &             1.60    & 0.00\% & 14\%\\
& 3    &         0.6573    &         0.0407    &             1.40    & 0.00\% & 14\%\\
& 4    &         0.3296    &         0.0396    &             1.00    & 0.00\% & 100\%\\
& 5    &         1.5159    &         0.0439    &             1.20    & 0.00\% & 31\%\\
& 6    &         0.2433    &         0.0436    &             1.00    & 0.00\% & 100\%\\
& 7    &         0.1892    &         0.0418    &             1.20    & 0.00\% & 31\%\\
& 8    &         0.1467    &         0.0434    &             1.40    & 0.00\% & 14\%\\
& 9    &         0.8045    &         0.0458    &             2.00    & 0.00\% & 7\%\\
& 10   &         0.5348    &         0.0557    &             1.40    & 0.00\% & 14\%\\
& 15   &         0.7455    &         0.0480    &             2.60    & 0.00\% & 3\%\\
& 20   &         3.2028    &         0.0650    &             2.20    & 0.00\% & 3\%\\
\hline
\multirow{11}{*}{708} & 2    &         0.9872    &         0.2452    &             1.20    & 0.00\% & 31\%\\
& 3    &         0.9582    &         0.2590    &             1.20    & 0.00\% & 31\%\\
& 4    &         0.8131    &         0.2700    &             1.00    & 0.00\% & 100\%\\
& 5    &         1.5395    &         0.3129    &             1.00    & 0.00\% & 100\%\\
& 6    &         2.8213    &         0.2791    &             1.00    & 0.00\% & 100\%\\
& 7    &         2.5991    &         0.5398    &             1.40    & 0.00\% & 14\%\\
& 8    &         1.6095    &         0.2723    &             1.00    & 0.00\% & 100\%\\
& 9    &         3.1732    &         0.2710    &             1.20    & 0.00\% & 31\%\\
& 10   &         1.5106    &         0.3644    &             1.00    & 0.00\% & 100\%\\
& 15   &         0.7631    &         0.1691    &             1.00    & 0.00\% & 100\%\\
& 20   &         0.3489    &         0.1693    &             1.00    & 30.77\% & 31\%\\
\hline
\multirow{11}{*}{818} & 2    &         1.3494    &         0.3060    &             1.20    & 0.00\% & 31\%\\
& 3    &       18.5982    &         0.2882    &             1.40    & 0.00\% & 14\%\\
& 4    &       23.5930    &         0.3012    &             1.60    & 0.00\% & 31\%\\
& 5    &         7.4189    &         0.3348    &             1.00    & 0.00\% & 100\%\\
& 6    &         3.2848    &         0.4409    &             1.20    & 0.00\% & 31\%\\
& 7    &         3.3584    &         0.3494    &             1.60    & 0.00\% & 14\%\\
& 8    &       30.8488    &         0.3733    &             1.20    & 0.00\% & 31\%\\
& 9    &       19.8565    &         0.3140    &             1.20    & 0.00\% & 31\%\\
& 10   &       34.5108    &         0.9278    &             1.40    & 0.00\% & 31\%\\
& 15   &         4.7518    &         0.2909    &             1.00    & 0.00\% & 100\%\\
& 20   &         0.8214    &         0.1928    &             1.00    & 0.00\% & 100\%\\
\hline
            \end{tabular}}
\end{center}
\caption{Average results for cardinality constrained problems.\label{table:1}}
\end{table}

\begin{table}[h]
\begin{center}
{\scriptsize \begin{tabular}{|r|r|rr|rrr|}\hline
$n$ & $p$ & CPUTime & CPUTime-Crisp & Different Pareto & CheckPareto & ReachIdeal\\\hline
\hline
\multirow{11}{*}{30} & 2     & 0.0026 & 0.0030 & 1.00  & 0.00\% & 100\%\\
& 3     & 0.0073 & 0.0024 & 1.00  & 0.00\% & 100\%\\
& 4     & 0.0068 & 0.0022 & 1.00  & 0.00\% & 100\%\\
& 5     & 0.0064 & 0.0022 & 1.00  & 0.00\% & 100\%\\
& 6     & 0.0058 & 0.0021 & 1.00  & 0.00\% & 100\%\\
& 7     & 0.0054 & 0.0021 & 1.00  & 0.00\% & 100\%\\
& 8     & 0.0053 & 0.0021 & 1.40  & 0.00\% & 14\%\\
& 9     & 0.0069 & 0.0028 & 1.60  & 0.00\% & 7\%\\
& 10    & 0.0051 & 0.0020 & 1.00  & 0.00\% & 100\%\\
& 15    & 0.0074 & 0.0013 & 1.60  & 0.00\% & 7\%\\
& 20    & 0.0020 & 0.0004 & 1.40  & 0.00\% & 14\%\\
\hline
\multirow{11}{*}{324} & 2     & 0.0540 & 0.0300 & 1.00  & 0.00\% & 100\%\\
& 3     & 0.1375 & 0.0511 & 1.20  & 0.00\% & 31\%\\
& 4     & 0.1676 & 0.0604 & 1.40  & 0.00\% & 31\%\\
& 5     & 0.1934 & 0.0561 & 1.40  & 0.00\% & 31\%\\
& 6     & 0.1746 & 0.0389 & 1.60  & 0.00\% & 7\%\\
& 7     & 0.2057 & 0.0332 & 1.60  & 0.00\% & 14\%\\
& 8     & 0.1919 & 0.0412 & 2.00  & 0.00\% & 7\%\\
& 9     & 0.2205 & 0.0462 & 2.20  & 0.00\% & 3\%\\
& 10    & 0.2431 & 0.0587 & 2.00  & 0.00\% & 3\%\\
& 15    & 0.4934 & 0.0659 & 1.80  & 0.00\% & 14\%\\
& 20    & 3.6875 & 0.1615 & 2.20  & 0.00\% & 3\%\\
\hline
\multirow{11}{*}{402} & 2     & 0.0782 & 0.0332 & 1.20  & 0.00\% & 31\%\\
& 3     & 0.1384 & 0.0502 & 1.00  & 0.00\% & 100\%\\
& 4     & 0.3412 & 0.0583 & 1.20  & 0.00\% & 31\%\\
& 5     & 0.4761 & 0.0589 & 1.20  & 0.00\% & 31\%\\
& 6     & 0.2379 & 0.0451 & 1.40  & 0.00\% & 14\%\\
& 7     & 0.2989 & 0.0483 & 1.20  & 0.00\% & 31\%\\
& 8     & 0.2809 & 0.0685 & 1.20  & 0.00\% & 14\%\\
& 9     & 0.1648 & 0.0409 & 1.20  & 0.00\% & 31\%\\
& 10    & 0.2939 & 0.0410 & 1.40  & 0.00\% & 14\%\\
& 15    & 0.6319 & 0.0483 & 2.00  & 0.00\% & 3\%\\
& 20    & 0.4941 & 0.0867 & 1.80  & 0.00\% & 14\%\\
\hline
\multirow{11}{*}{500} & 2     & 0.1286 & 0.0456 & 1.20  & 0.00\% & 31\%\\
& 3     & 0.5093 & 0.0550 & 1.20  & 0.00\% & 31\%\\
& 4     & 0.7750 & 0.0805 & 1.20  & 0.00\% & 31\%\\
& 5     & 0.7108 & 0.1182 & 1.20  & 0.00\% & 31\%\\
& 6     & 0.2592 & 0.0785 & 1.00  & 0.00\% & 100\%\\
& 7     & 0.1792 & 0.0585 & 1.00  & 0.00\% & 100\%\\
& 8     & 0.1882 & 0.0729 & 1.00  & 0.00\% & 100\%\\
& 9     & 0.3244 & 0.0522 & 1.20  & 0.00\% & 31\%\\
& 10    & 0.3196 & 0.0648 & 1.80  & 0.00\% & 3\%\\
& 15    & 0.2469 & 0.0610 & 1.00  & 0.00\% & 100\%\\
& 20    & 0.6113 & 0.0561 & 1.20  & 0.00\% & 31\%\\
\hline
\multirow{11}{*}{708} & 2     & 0.4718 & 0.1827 & 1.00  & 0.00\% & 100\%\\
& 3     & 8.7481 & 0.4953 & 1.00  & 0.00\% & 100\%\\
& 4     & 43.5878 & 0.8488 & 1.00  & 0.00\% & 100\%\\
& 5     & 12.1678 & 0.9080 & 1.00  & 0.00\% & 100\%\\
& 6     & 11.4501 & 1.0463 & 1.00  & 0.00\% & 100\%\\
& 7     & 20.8775 & 1.9400 & 1.00  & 0.00\% & 31\%\\
& 8     & 6.3216 & 2.1380 & 1.00  & 0.00\% & 100\%\\
& 9     & 34.5141 & 2.1440 & 1.40  & 0.00\% & 14\%\\
& 10    & 8.4178 & 2.3144 & 1.20  & 0.00\% & 31\%\\
& 15    & 1.9882 & 0.1593 & 1.20  & 7.69\% & 31\%\\
& 20    & 0.2744 & 0.0929 & 1.00  & 0.00\% & 100\%\\
\hline
\multirow{11}{*}{818} & 2     & 0.5556 & 0.2972 & 1.00  & 0.00\% & 100\%\\
& 3     & 32.1120 & 0.4418 & 1.25  & 0.00\% & 25\%\\
& 4     & 54.7338 & 0.6886 & 1.25  & 0.00\% & 25\%\\
& 5     & 33.4988 & 0.7135 & 1.25  & 0.00\% & 25\%\\
& 6     & 34.0156 & 1.4136 & 2.25  & 0.00\% & 0\%\\
& 7     & 54.7128 & 2.1214 & 1.50  & 0.00\% & 10\%\\
& 8     & 12.9773 & 1.9467 & 1.50  & 0.00\% & 10\%\\
& 9     & 26.0821 & 2.7770 & 1.75  & 0.00\% & 4\%\\
& 10    & 13.0448 & 1.8998 & 1.50  & 0.00\% & 10\%\\
& 15    & 4.1847 & 0.2182 & 1.25  & 20.00\% & 10\%\\
& 20    & 0.5317 & 0.0965 & 1.00  & 0.00\% & 100\%\\
\hline
\end{tabular}}
\end{center}
\caption{Average results for budgeted problems.\label{table:2}}
\end{table}

\begin{table}[h]
\begin{center}
{\scriptsize \begin{tabular}{|c|c|cccc|c|}
\hline
$n$&$B$&  \%CoveredCrisp & \%CoveredFuzzy$^-$ & \%CoveredFuzzy & \%CoveredFuzzy$^+$ & \#OpenFuzzy  \\
\hline\hline
\multirow{11}{*}{30} & 2    & 66.18\% & 46.67\% & 52.19\% & 58.07\% & 1.3\\
& 3    & 71.85\% & 62.01\% & 68.16\% & 75.80\% & 2.9\\
& 4    & 76.05\% & 66.22\% & 73.09\% & 80.62\% & 3.8\\
& 5    & 79.34\% & 69.34\% & 76.56\% & 84.49\% & 4.8\\
& 6    & 82.08\% & 72.00\% & 79.56\% & 87.79\% & 5.8\\
& 7    & 84.28\% & 75.35\% & 83.33\% & 91.83\% & 7.0\\
& 8    & 86.47\% & 75.78\% & 84.68\% & 93.54\% & 8.0\\
& 9    & 88.48\% & 78.34\% & 86.40\% & 95.25\% & 8.8\\
& 10   & 90.31\% & 81.52\% & 88.30\% & 97.62\% & 10.0\\
& 15   & 97.81\% & 85.66\% & 96.93\% & 106.46\% & 14.8\\
& 20   & 100.00\% & 90.53\% & 100.00\% & 109.96\% & 18.2\\
\hline
\multirow{11}{*}{324} & 2    & 21.71\% & 17.93\% & 20.08\% & 21.76\% & 2.0\\
& 3    & 28.77\% & 23.65\% & 26.54\% & 29.04\% & 3.0\\
& 4    & 35.30\% & 29.48\% & 33.13\% & 35.96\% & 4.0\\
& 5    & 41.54\% & 34.28\% & 38.51\% & 41.88\% & 5.0\\
& 6    & 47.74\% & 39.17\% & 43.48\% & 47.95\% & 6.0\\
& 7    & 52.83\% & 43.69\% & 48.62\% & 53.10\% & 7.0\\
& 8    & 57.54\% & 48.79\% & 54.35\% & 59.56\% & 8.0\\
& 9    & 61.92\% & 51.77\% & 58.01\% & 63.33\% & 9.0\\
& 10   & 66.00\% & 55.19\% & 61.88\% & 67.38\% & 10.0\\
& 15   & 82.64\% & 67.82\% & 75.63\% & 82.92\% & 15.0\\
& 20   & 93.46\% & 78.49\% & 87.71\% & 95.83\% & 20.0\\
\hline
\multirow{11}{*}{402} & 2    & 19.39\% & 12.45\% & 13.64\% & 15.33\% & 1.6\\
& 3    & 27.61\% & 20.01\% & 22.13\% & 24.81\% & 3.0\\
& 4    & 34.23\% & 26.54\% & 29.13\% & 32.66\% & 3.7\\
& 5    & 39.60\% & 31.30\% & 34.58\% & 38.59\% & 4.6\\
& 6    & 44.96\% & 35.88\% & 39.36\% & 43.73\% & 5.5\\
& 7    & 49.70\% & 38.93\% & 42.64\% & 47.34\% & 6.2\\
& 8    & 54.41\% & 43.00\% & 47.24\% & 52.38\% & 7.2\\
& 9    & 58.28\% & 48.51\% & 53.39\% & 59.05\% & 8.5\\
& 10   & 61.86\% & 48.90\% & 54.70\% & 60.32\% & 9.2\\
& 15   & 77.01\% & 66.07\% & 72.83\% & 80.42\% & 14.5\\
& 20   & 87.86\% & 73.63\% & 81.83\% & 90.55\% & 19.2\\
\hline
\multirow{11}{*}{500} & 2    & 15.73\% & 11.80\% & 13.28\% & 14.69\% & 1.9\\
& 3    & 22.07\% & 15.82\% & 17.76\% & 19.53\% & 2.5\\
& 4    & 28.37\% & 22.41\% & 24.88\% & 27.60\% & 3.6\\
& 5    & 33.40\% & 25.11\% & 28.06\% & 30.97\% & 4.2\\
& 6    & 37.74\% & 31.47\% & 35.15\% & 38.79\% & 5.6\\
& 7    & 41.92\% & 37.23\% & 41.01\% & 45.16\% & 6.8\\
& 8    & 46.06\% & 40.14\% & 44.64\% & 49.13\% & 7.9\\
& 9    & 49.77\% & 41.69\% & 46.55\% & 51.49\% & 8.4\\
& 10   & 52.95\% & 45.18\% & 50.20\% & 55.29\% & 9.5\\
& 15   & 66.28\% & 56.09\% & 62.44\% & 68.92\% & 14.5\\
& 20   & 76.51\% & 65.05\% & 72.42\% & 79.89\% & 19.3\\
\hline
\multirow{11}{*}{708} & 2    & 52.21\% & 42.73\% & 47.24\% & 51.97\% & 2.0\\
& 3    & 67.25\% & 59.35\% & 65.70\% & 72.51\% & 3.0\\
& 4    & 79.46\% & 69.41\% & 76.74\% & 84.33\% & 4.0\\
& 5    & 85.87\% & 75.21\% & 83.20\% & 91.49\% & 5.0\\
& 6    & 90.15\% & 79.50\% & 87.94\% & 96.70\% & 6.0\\
& 7    & 93.35\% & 83.74\% & 92.53\% & 102.05\% & 7.0\\
& 8    & 96.40\% & 85.77\% & 94.91\% & 104.39\% & 8.0\\
& 9    & 98.29\% & 88.32\% & 97.69\% & 107.85\% & 9.0\\
& 10   & 99.66\% & 89.41\% & 98.95\% & 108.85\% & 10.0\\
& 15   & 100.00\% & 90.35\% & 100.00\% & 110.02\% & 14.6\\
& 20   & 100.00\% & 90.40\% & 100.00\% & 110.41\% & 19.1\\
\hline
\multirow{11}{*}{818} & 2    & 43.30\% & 33.55\% & 37.19\% & 41.25\% & 2.0\\
& 3    & 57.27\% & 46.32\% & 51.50\% & 56.46\% & 3.0\\
& 4    & 69.75\% & 56.62\% & 62.89\% & 69.67\% & 4.0\\
& 5    & 79.88\% & 68.00\% & 75.38\% & 82.96\% & 5.0\\
& 6    & 84.54\% & 73.35\% & 81.15\% & 89.67\% & 6.0\\
& 7    & 89.15\% & 78.51\% & 86.77\% & 95.55\% & 7.0\\
& 8    & 92.70\% & 79.27\% & 87.99\% & 97.25\% & 8.0\\
& 9    & 95.68\% & 82.73\% & 92.09\% & 101.11\% & 9.0\\
& 10   & 97.38\% & 84.88\% & 94.20\% & 104.13\% & 9.3\\
& 15   & 100.00\% & 90.05\% & 99.92\% & 109.92\% & 14.6\\
& 20   & 100.00\% & 90.12\% & 100.00\% & 110.01\% & 19.2\\\hline
\end{tabular}}
\end{center}
\caption{Average coverage results for cardinality-constrained problems.\label{table:3}}
\end{table}

\begin{table}[h]
\begin{center}
{\scriptsize \begin{tabular}{|c|c|cccc|cc|}
\hline
$n$&$p$&  \%CoveredCrisp & \%CoveredFuzzy$^-$ & \%CoveredFuzzy & \%CoveredFuzzy$^+$ & \#OpenCrisp & \#OpenFuzzy  \\
\hline\hline
\multirow{11}{*}{30} & 2     & 58.72\% & 48.40\% & 54.04\% & 58.40\% & 1.6   & 1.4\\
& 3     & 67.82\% & 57.13\% & 63.77\% & 69.03\% & 2.4   & 2.2\\
& 4     & 73.20\% & 62.22\% & 69.51\% & 75.20\% & 3.6   & 3.0\\
& 5     & 77.00\% & 65.98\% & 73.71\% & 79.83\% & 4.6   & 4.0\\
& 6     & 79.78\% & 69.02\% & 77.00\% & 83.34\% & 5.2   & 5.0\\
& 7     & 82.41\% & 71.57\% & 79.89\% & 86.41\% & 6.2   & 6.0\\
& 8     & 84.35\% & 74.49\% & 82.97\% & 89.54\% & 7.0   & 7.0\\
& 9     & 86.70\% & 75.85\% & 84.01\% & 91.29\% & 8.3   & 8.0\\
& 10    & 88.48\% & 77.43\% & 86.58\% & 93.70\% & 9.0   & 9.0\\
& 15    & 96.53\% & 85.45\% & 94.64\% & 102.19\% & 14.0  & 14.0\\
& 20    & 100.00\% & 87.52\% & 99.06\% & 107.74\% & 18.0  & 17.6\\
\hline
\multirow{11}{*}{324} & 2     & 14.00\% & 11.73\% & 12.99\% & 14.27\% & 1.4   & 1.0\\
& 3     & 22.98\% & 19.09\% & 21.13\% & 23.04\% & 2.8   & 2.0\\
& 4     & 28.82\% & 25.38\% & 28.13\% & 30.79\% & 3.7   & 3.0\\
& 5     & 35.30\% & 31.10\% & 34.61\% & 38.19\% & 4.0   & 4.0\\
& 6     & 41.54\% & 35.74\% & 40.20\% & 44.38\% & 5.0   & 5.0\\
& 7     & 47.74\% & 41.65\% & 46.58\% & 51.50\% & 6.0   & 6.0\\
& 8     & 52.83\% & 45.94\% & 51.22\% & 56.46\% & 7.0   & 7.0\\
& 9     & 57.54\% & 49.56\% & 55.29\% & 61.03\% & 8.0   & 8.0\\
& 10    & 61.92\% & 53.45\% & 59.55\% & 65.75\% & 9.0   & 9.0\\
& 15    & 79.93\% & 69.19\% & 76.58\% & 84.74\% & 14.0  & 14.0\\
& 20    & 91.50\% & 78.52\% & 87.55\% & 96.55\% & 19.0  & 18.3\\
\hline
\multirow{11}{*}{402} & 2     & 13.02\% & 9.05\% & 10.24\% & 11.31\% & 1.9   & 1.1\\
& 3     & 20.05\% & 16.82\% & 18.53\% & 20.41\% & 2.4   & 2.0\\
& 4     & 28.08\% & 23.17\% & 25.91\% & 28.31\% & 3.7   & 3.0\\
& 5     & 35.63\% & 28.45\% & 31.74\% & 34.68\% & 4.7   & 4.0\\
& 6     & 39.62\% & 33.68\% & 37.26\% & 41.04\% & 5.0   & 5.0\\
& 7     & 45.02\% & 37.07\% & 40.67\% & 44.89\% & 6.1   & 6.0\\
& 8     & 49.70\% & 41.08\% & 45.36\% & 49.99\% & 7.0   & 7.0\\
& 9     & 54.41\% & 44.43\% & 49.24\% & 54.22\% & 8.0   & 8.0\\
& 10    & 58.28\% & 49.44\% & 54.70\% & 60.21\% & 9.0   & 9.0\\
& 15    & 74.29\% & 62.30\% & 68.87\% & 75.95\% & 14.0  & 14.0\\
& 20    & 86.13\% & 71.87\% & 79.54\% & 87.51\% & 19.0  & 19.0\\
\hline
\multirow{11}{*}{500} & 2     & 12.08\% & 7.66\% & 8.34\% & 8.93\% & 2.0   & 1.2\\
& 3     & 18.84\% & 13.90\% & 15.15\% & 16.67\% & 2.9   & 2.0\\
& 4     & 23.75\% & 19.02\% & 20.89\% & 23.00\% & 3.8   & 3.0\\
& 5     & 28.37\% & 23.90\% & 26.32\% & 29.12\% & 4.0   & 4.0\\
& 6     & 33.40\% & 28.88\% & 31.84\% & 35.18\% & 5.0   & 5.0\\
& 7     & 37.74\% & 33.12\% & 36.58\% & 40.55\% & 6.0   & 6.0\\
& 8     & 41.92\% & 36.66\% & 40.58\% & 44.95\% & 7.0   & 7.0\\
& 9     & 46.06\% & 39.52\% & 43.73\% & 48.51\% & 8.0   & 8.0\\
& 10    & 49.77\% & 42.91\% & 47.49\% & 52.47\% & 9.0   & 9.0\\
& 15    & 63.83\% & 54.42\% & 60.21\% & 66.59\% & 14.0  & 14.0\\
& 20    & 74.82\% & 64.52\% & 71.82\% & 79.12\% & 19.0  & 18.9\\
\hline
\multirow{11}{*}{708} & 2     & 48.16\% & 38.25\% & 42.22\% & 46.55\% & 2.0   & 2.0\\
& 3     & 62.08\% & 50.71\% & 56.07\% & 61.75\% & 3.0   & 3.0\\
& 4     & 73.83\% & 60.07\% & 66.49\% & 73.22\% & 4.0   & 3.8\\
& 5     & 81.76\% & 68.54\% & 75.96\% & 83.53\% & 4.8   & 4.6\\
& 6     & 86.69\% & 74.14\% & 82.19\% & 90.42\% & 5.6   & 5.2\\
& 7     & 90.45\% & 80.11\% & 88.38\% & 97.26\% & 6.8   & 6.7\\
& 8     & 93.49\% & 81.54\% & 90.44\% & 99.51\% & 7.2   & 7.0\\
& 9     & 96.40\% & 84.49\% & 93.54\% & 103.08\% & 8.0   & 8.0\\
& 10    & 98.29\% & 87.26\% & 96.80\% & 106.28\% & 9.0   & 9.0\\
& 15    & 100.00\% & 90.56\% & 100.00\% & 110.04\% & 14.0  & 14.0\\
& 20    & 100.00\% & 90.14\% & 100.00\% & 110.24\% & 18.0  & 18.8\\
\hline
\multirow{11}{*}{818} & 2     & 37.08\% & 27.46\% & 30.59\% & 33.57\% & 1.8   & 1.5\\
& 3     & 51.64\% & 41.94\% & 46.81\% & 51.72\% & 2.9   & 2.8\\
& 4     & 64.17\% & 49.68\% & 55.60\% & 61.06\% & 3.9   & 3.2\\
& 5     & 74.40\% & 59.90\% & 66.91\% & 73.58\% & 4.9   & 4.1\\
& 6     & 80.09\% & 68.74\% & 76.58\% & 84.24\% & 5.5   & 5.0\\
& 7     & 84.94\% & 74.48\% & 83.04\% & 91.58\% & 6.5   & 6.0\\
& 8     & 89.15\% & 78.60\% & 87.65\% & 96.60\% & 7.0   & 7.0\\
& 9     & 92.72\% & 81.12\% & 90.26\% & 99.34\% & 8.3   & 8.0\\
& 10    & 95.68\% & 84.69\% & 94.43\% & 104.11\% & 9.0   & 9.0\\
& 15    & 100.00\% & 89.53\% & 99.71\% & 109.46\% & 14.0  & 14.0\\
& 20    & 100.00\% & 89.76\% & 100.00\% & 110.00\% & 19.0  & 18.5\\\hline
\end{tabular}}
\end{center}
\caption{Average coverage results for budgeted problems.\label{table:4}}
\end{table}

The results of Tables \ref{table:3} and \ref{table:4} indicate that the covered demands and the number of open plants are quite similar in the cardinality constrained and budget constrained models. Nevertheless, the results of these tables also indicate that the fuzzy models are more restrictive than the crisp ones, both in terms of the covered demand and the number of open plants. As explained when discussing Example 1, this is due to the definition of the coverage sets $\tilde J$ as well as to the three cardinality/budget constraints. We must however recall that the fuzzy model is more general than the crisp one, as it deals with imprecise knowledge of all data involved.

\section{Conclusions}\label{sec:conclu}
In this paper we propose a new and general model for the maximal covering location with imprecise knowledge on all data, by means of fuzzy numbers and variables. The properties of the proposed model allow to formulate it with an equivalent mixed-binary linear multiobjective programme. For obtaining fuzzy Pareto solutions we propose a solution algorithm based on an augmented weighted Tchebycheff method. The 
 effectiveness of the proposed solution algorithm has been confirmed by extensive computational experiments whose numerical results are presented and analyzed.\\
Promising avenues for future research include the study from a fuzzy perspective of other classical discrete location models, like the plant location problem. From a different point of view, further insight for dealing with imprecise knowledge with these types of problems can be derived from the study of more sophisticated fuzzy sets like, for instance, intuitionistic numbers.\\

\section*{Acknowledgments}

This work was supported by the Spanish Ministry of Economy and Competitiveness through MINECO/FEDER grants MTM2017-89577-P, MTM2015-63779-R, MTM2016-
74983-C2-1-R. This support is gratefully acknowledged.

\end{document}